\pdfoutput=1
\documentclass[11pt]{elsarticle}
\usepackage{amsmath,amssymb,latexsym}
\usepackage{overpic}
\usepackage{esint}
\usepackage{pifont}
\newtheorem{thm}{Theorem}

\newtheorem{lemma}[thm]{Lemma}
\newtheorem{proposition}[thm]{Proposition}

\newdefinition{definition}{Definition}
\newdefinition{remark}{Remark}
\newproof{proof}{Proof}
\newproof{pot}{Proof of Theorem \ref{thm2}}

\newcommand{\field}[1]{\mathbb{#1}}
\newcommand{\R}{\field{R}}

\newcommand{\N}{\field{N}}
\newcommand{\C}{\field{C}}

\newcommand{\CC}{{\mathcal C}}
\newcommand{\OO}{{\mathcal O}}

\newcommand{\NN}{{\mathbf N}}
\newcommand{\PP}{{\mathbf P}}

\renewcommand{\Re}{\mathop{\rm Re}}
\renewcommand{\Im}{\mathop{\rm Im}}

\newcommand{\isdef}{\stackrel{\text{\tiny def}}{=}}

\begin{document}
\begin{frontmatter}

\title{On a conjecture of A.~Magnus concerning the asymptotic behavior of the recurrence coefficients of the
generalized Jacobi polynomials}

\author{A.~Foulqui\'{e} Moreno}
\ead{foulquie@ua.pt}
\address{Department of Mathematics and Research Unity ``Matem\'{a}tica e Aplica\c{c}\~{o}es", University of Aveiro,\\Campus Universit\'ario de Santiago, 3810-193 Aveiro, Portugal}
\author{A.~Mart\'{\i}nez-Finkelshtein}
\ead{andrei@ual.es}
\address{Department of Statistics and Applied Mathematics,
University of Almer\'{\i}a, SPAIN, and\\
Instituto Carlos I de F\'{\i}sica Te\'{o}rica y Computacional,
Granada University, Spain}
\author{V.L. Sousa\corref{cor1}}
\ead{vsousa@ua.pt}
\address{EBI/JI Dr. Manuel M. Machado and Research Unity ``Matem\'{a}tica e Aplica\c{c}\~{o}es", University of Aveiro,\\Campus Universit\'ario de Santiago, 3810-193 Aveiro, Portugal}

\cortext[cor1]{Corresponding author}

\begin{abstract}
In 1995 Magnus \cite{Magnus1995} posed a conjecture about the asymptotics of the recurrence coefficients of orthogonal polynomials with respect to the
weights on $[-1,1]$ of the form
\begin{equation*}
  \left(  1-x\right)  ^{\alpha} \left(  1+x\right)  ^{\beta}\left|
x_{0}-x\right|  ^{\gamma} \times
\begin{cases}
B, & \text{for }x\in\left[  -1,x_{0}\right)
\text{,}\\
A, & \text{for }x\in\left[  x_{0},1\right]  \text{,}%
\end{cases}
\end{equation*}
with $A, B>0$, $\alpha, \beta, \gamma>-1$, and $x_{0}\in\left(-1,1\right)  $.  We show rigorously that Magnus' conjecture is correct even in a more general situation, when the weight above has an extra factor, which is analytic in a neighborhood of $[-1,1]$ and positive on the interval. The proof is based on the steepest descendent method of Deift and Zhou applied to the non-commutative Riemann-Hilbert problem characterizing the orthogonal polynomials. A feature of this situation is that the local analysis at $x_0$ has to be carried out in terms of confluent hypergeometric functions.
\end{abstract}


\end{frontmatter}

\section{Introduction and statement of results}

\subsection{Introduction} \label{sec:intro}

A.~Magnus considered in \cite{Magnus1995} a weight function which is smooth and positive on the whole
interval of orthogonality up to a finite number of points where algebraic singularities occur. His primary goal was to investigate the influence of these singular points on the asymptotic behavior of the recurrence coefficients of the corresponding orthogonal polynomials (generalized Jacobi polynomials). 
Based on numerical evidence, he conjectured explicit formulas for the asymptotics of these coefficients (as the degree of the polynomial grows) for the weights of the form
\begin{equation}\label{wheight-Magnus}
  \left(  1-x\right)  ^{\alpha} \left(  1+x\right)  ^{\beta}\left|
x_{0}-x\right|  ^{\gamma} \times
\begin{cases}
B, & \text{for }x\in\left[  -1,x_{0}\right)
\text{,}\\
A, & \text{for }x\in\left[  x_{0},1\right]  \text{,}%
\end{cases}
\end{equation}
with $A, B>0$ and $\alpha, \beta, \gamma>-1$, and $x_{0}\in\left(-1,1\right)  $. This weight combines at a single point both an algebraic singularity and a jump.

So far, Magnus' conjecture has been confirmed rigorously in some special cases (see below);
our main goal is to establish it in its full generality, and even to extend it further. Namely, we consider polynomials that are orthogonal on a finite interval $[-1,1]$ with respect to a modified Jacobi weight of the form
 \begin{equation} \label{Definitiew}
        w_{c,\gamma}(x) = (1-x)^{\alpha} (1+x)^{\beta} |x_0-x|^{\gamma} h(x)\, \Xi_{c}(x), \qquad x \in [-1,1],
    \end{equation}%
where $x_0 \in (-1,1)$, $\alpha, \beta, \gamma > -1$, $h$ is real analytic and strictly positive on $[-1,1]$, and $ \Xi_{c}$ is a step-like function, equal to $1$ on $[-1,x_0)$ and $c^2>0$ on $[x_0,1]$.

The proof is based on the nonlinear steepest descent analysis of Deift and Zhou, introduced in \cite{MR94d:35143} and further developed in \cite{MR2001m:05258a, MR98b:35155, MR96d:34004}, which is based on the Riemann--Hilbert characterization of orthogonal polynomials due to Fokas, Its, and
Kitaev \cite{Fokas92}. A crucial contribution to this approach is \cite{MR2087231}, where the complete asymptotic expansion for the orthogonal polynomials with respect to the Jacobi weight modified by a real analytic and strictly positive function was obtained (in notation \eqref{wheight-Magnus}, $A=B$ and $\gamma=0$). The first application of this technique to weights with a jump discontinuity is due to \cite{Its07b}, where the authors considered an exponential weight on $\R$ with a jump at the origin.

Let $P_n(x) = P_n(x;w_{c,\gamma})$ be the monic polynomial of degree $n$ %
orthogonal with respect to the weight $w_{c,\gamma}$ on $[-1,1]$,
    \[
        \int_{-1}^1 P_n(x;w_{c,\gamma}) x^k w_{c,\gamma}(x) \ dx = 0, \qquad
        \mbox{for $k = 0,1, \dots, n-1$.}
    \]
It is well known (see e.g.~\cite{szego:1975}) that $\{P_n\}$
satisfy the three-term recurrence relation
    \begin{equation} \label{RecRelphin}
        P_{n+1}(z) = (z-b_n) P_n(z) - a_n^2 P_{n-1}(z).
    \end{equation}
The central result of this paper is:
\begin{thm} \label{theoremRecCoef}
The recurrence coefficients $a_n$ and $b_n$ of orthogonal polynomials corresponding to the %
generalized Jacobi weight \eqref{Definitiew} have a complete asymptotic expansion of the form%
$$
 a_n  = \dfrac 12 -\sum_{k=1}^{\infty}\frac{A_k(n)}{n^k}, \quad  b_n = -\sum_{k=1}^{\infty}\frac{B_k(n)}{n^k},
$$\label{asymptotics abn}%
as $n\to \infty$, where for every $k\in \N$ the coefficients $A_k(n)$ and $B_k(n)$ are bounded in $n$. In particular,
  \begin{align} \label{asymptotics An}
    A_1(n) &= -\dfrac {\sqrt{1-x_{0}^2}}{2} \sqrt{ \dfrac {\gamma^2}{4}+\dfrac{\log^2c}{\pi^2}}%
    \cos\left[2n \arccos x_0+2\dfrac{\log c}{\pi}\log \left(4n\sqrt{1-x_{0}^2} \right)-\Theta\right],\\%
  \label{asymptotics Bn}
    B_1(n) &= -\sqrt{1-x_{0}^2} \sqrt{ \dfrac {\gamma^2}{4}+\dfrac{\log^2c}{\pi^2}} \cos\left[(2n+1)\arccos x_0+2\dfrac{\log c}{\pi} \log \left(4n\sqrt{1-x_{0}^2} \right)-\Theta\right],
         \end{align}%
where
    \begin{equation}\label{definitionTheta}
    \begin{split}
    \Theta= \left( \alpha+\dfrac {\gamma} {2} \right) \pi- &\left(\alpha+\beta+\gamma \right)\arccos x_0 - 2\arg\Gamma \left(\dfrac\gamma 2-i\dfrac{\log c}{\pi} \right) \\& - \arg\left(\dfrac \gamma2-i\dfrac{\log c}{\pi} \right) -\frac{\sqrt{1-x_{0}^{2}}}{\pi}\fint_{-1}^{1}\frac{\log h\left(  t\right)
}{\sqrt{1-t^{2}}}\frac{dt}{t-x_{0}},
     \end{split}
     \end{equation}%
and $\fint$ denotes the integral understood in terms of its principal value.
\end{thm}

\begin{remark}
 We can rewrite the result of this theorem as
  \begin{equation*}
    a_n = \dfrac{1}{2} -\dfrac{M}n \cos\left[2n \arccos x_0-2\mu\log \left(4n\sqrt{1-x_{0}^2} \right)-\Theta\right]
         +\OO\left(  \frac{1}{n^{2}} \right),%
  \end{equation*}
    \begin{equation*}
    b_n =  -\dfrac{2M}n \cos\left[(2n+1)\arccos x_0-2\mu\log \left(4n\sqrt{1-x_{0}^2} \right)-\Theta\right]
         +\OO\left(  \frac{1}{n^{2}} \right),
         \end{equation*}%
as $n\to\infty$, where%
   \begin{equation*}\label{definitionmu}
    \mu=-\dfrac{\log c}{\pi}, \, \, \, \,
        M=\dfrac {\sqrt{1-x_{0}^2}}{2} \sqrt{ \dfrac {\gamma^2}{4}+\mu^2} ,
    \end{equation*}%
and $\Theta$ defined by \eqref{definitionTheta}. %

A comparison of these formulas with those in \cite{Magnus1995} (setting $h(x)\equiv B$ and $c^2=A/B$) shows that Magnus' conjecture on the asymptotic behavior of the recurrence coefficients holds true for weights of the form \eqref{wheight-Magnus}. Observe that this is a slight extension of the original statement of Magnus: (i) we allow for an extra real analytic and strictly positive  factor $h$ in the weight, and (ii) we can replace the error term $o(1/n)$ in \cite{Magnus1995}  by a more precise $\OO(1/n^2)$.
\end{remark}

Theorem  \ref{theoremRecCoef} generalizes some previous known results about the asymptotics of the recurrence coefficients. To mention a few, weight $w_{1,0}$ was considered in \cite{MR2087231}, $w_{1,\gamma}$  is a particular case of the weight studied in \cite{Van:2003}, and $w_{c,0}$ was matter of attention in \cite{FMS}.

The proof is based on  the steepest descent analysis of the non-commutative Riemann-Hilbert problem characterizing the orthogonal polynomials $P_{n}$.  In theory, this approach allows to compute all coefficients $A_{k}$ and $B_{k}
$ in \eqref{asymptotics An}--\eqref{asymptotics Bn}. However, the computations are cumbersome and their complexity increases with $k$.

Most of the steps of the steepest descent analysis below are standard and well described in the literature, see e.g. \cite{MR2001f:42037, MR2000g:47048, MR2087231}. The main feature of the situation treated here in comparison with the classical Jacobi weight is the singularity of the weight of orthogonality at $x_0$. 
The local behavior of $P_{n}$'s at $x_{0}$ is described in terms of the confluent hypergeometric functions. Such functions appeared already in the Riemann-Hilbert analysis for the weight $w_{c,0}$ in \cite{FMS} and \cite{Its07b}, and will work also in the general situation studied here. However, the parameter describing the family of these functions is complex; its real part depends on the degree of the algebraic singularity $\gamma$ and its imaginary part is a function of the size of the jump $c^{2}$ in the weight $w_{c,\gamma}$. Also, for $c=1$ these confluent hypergeometric functions degenerate into the Bessel functions, in correspondence with \cite{MR2087231}.

Interestingly enough, the confluent hypergeometric functions appear in the scaling limit (as the number of particles goes to infinity) of the correlation functions of the pseudo-Jacobi ensemble, see  \cite{Borodin:2001xr}. This ensemble corresponds to a sequence of weights of the form
\begin{equation}\label{borodin}
(1+x^{2})^{-n-\Re(s)} e^{2\Im(s) \arg(1+i x)}, \quad x\in \R, 
\end{equation}
where $n$ is the degree of the polynomial and $s$ a complex parameter. The connection between both problems becomes apparent if we perform the inversion $x\mapsto 1/x$ in \eqref{borodin}; this creates at the origin an algebraic singularity with the exponent $\Re(s)$ and a jump depending on $\Im(s)$.

In the next Section we state the Riemann-Hilbert problem for the orthogonal polynomials and perform the steepest descent analysis; as a result, Theorem \ref{theoremRecCoef} is proved in Section 3. For  the sake of brevity, the description of the standard steps is rather sketchy; an interested reader may find the missing details in the literature cited above, and especially in  \cite{FMS}. However, the local parametrix (the Riemann-Hilbert problem in a neighborhood of the singularity) at $x_{0}$ is analyzed in full detail. The same problem has appeared very recently in an independent work of Deift, Its and Krasovsky \cite{Deift:oe} on the asymptotics of  {T}oeplitz and {H}ankel determinants.   

\section{The steepest descent analysis} \label{sec:RHanalysis}

\subsection{The Riemann-Hilbert problem and first transformations}

Following Fokas, Its and Kitaev \cite{Fokas92}, we characterize the orthogonal polynomials in terms of the unique solution $\mathbf Y$
of the following $2\times2$ matrix valued Riemann-Hilbert (RH) problem: for $n\in \N$,
\begin{enumerate}
\item[(Y1)] $\mathbf Y$ is analytic in $\C\setminus [-1,1]$.
\item[(Y2)] On $(-1,x_0) \cup (x_0,1)$, $\mathbf Y$ possesses continuous boundary values $\mathbf Y_+$ (from the upper half
plane) and $\mathbf Y_-$ (from the lower half plane), and
\[
\mathbf Y_{+}(x)=\mathbf Y_{-}(x)\,
\begin{pmatrix}
1 & w _{c,\gamma}  (x) \\
0 & 1
\end{pmatrix}.
\]
\item[(Y3)] $\mathbf  Y(z)=\left(  \mathbf{I}+\OO\left( 1/z \right)  \right) \, z^{n \sigma_{3}}$, as $z\rightarrow\infty$,
where all terms are $2 \times 2 $ matrices,  $\mathbf I$ is the identity  and  $\sigma_3=\begin{pmatrix}
1 & 0 \\ 0 & -1
\end{pmatrix}$ is the third Pauli matrix.\footnote{In what follows, for $a\in \C \setminus \{0\}$ and  $b\in \C$, $a^{ b\sigma_{3}}$  we use the notation
$$
a^{b \sigma_{3}} \isdef \begin{pmatrix}
a^{b}  & 0\\
0 & 1/a^{b}
\end{pmatrix}.
$$
}
\item[(Y4)]  if $(\zeta,s)\in \{(-1,\beta), (x_0,\gamma), (1,\alpha)\}$ then for $z\rightarrow \zeta $, $z\in\mathbb{C}\backslash\left[  -1,1\right]  $,
\[
\mathbf Y(z)= \begin{cases}
\mathcal O
\begin{pmatrix}
1 & \left\vert z-\zeta\right\vert ^{s}\\
1 & \left\vert z-\zeta\right\vert ^{s}%
\end{pmatrix}, & \text{if }s<0;\\
\mathcal O \begin{pmatrix}
1 & \log\left\vert z-\zeta\right\vert \\
1 & \log\left\vert z-\zeta\right\vert
\end{pmatrix},
  & \text{if }s=0;\\
\mathcal O
\begin{pmatrix}
1 & 1\\
1 & 1
\end{pmatrix},
& \text{if }s>0.%
\end{cases}
\]
\end{enumerate}%
Standard arguments (see e.g.~\cite{MR2087231}) show that this RH problem has a unique solution given by
\begin{equation}
\mathbf  Y\left(  z, n\right)  =\left(
\begin{array}
[c]{cc}%
P_{n}\left(  z\right)  & \CC\left(  P_{n}w_{c,\gamma}\right)  \left(  z\right) \\
-2\pi ik_{n-1}^{2}P_{n-1}\left(  z\right)  & -2\pi ik_{n-1}^{2}\CC\left(
P_{n-1}w_{c,\gamma}\right)  \left(  z\right)
\end{array}
\right)  \text{,} \label{sol Y}%
\end{equation}
where $P_{n}$ is the monic orthogonal polynomial of degree $n$\ with respect
to $w_{c,\gamma}$, $p_n(x) = p_n(x;w_{c,\gamma})$ is the corresponding orthonormal polynomial,
    \[
        p_n(x) = k_n P_n(x),
    \]
where $k_n > 0$ is the leading coefficient of $p_n$, %
and $\CC\left(  \cdot \right)  $\ is the Cauchy transform on $\left[  -1,1\right]  $
defined by
\[
\CC\left(  f\right)  \left(  z\right)  =\frac{1}{2\pi i}\int_{-1}^{1}%
\frac{f\left(  x\right)  }{x-z}\, dx\,.
\]
Note that $\mathbf Y$ and other matrices introduced hereafter depend on $n$; however, to simplify notation we omit the explicit reference to $n$.

The first transformations of the Deift-Zhou steepest descendent method are standard, and up to slight variations match those described in %
\cite[subsection 2.2]{FMS}. Hence, we will omit the details, highlighting basically the differences with the cited reference. The reader should keep in mind also that in \cite{FMS} the analysis is made for a singularity fixed at $x_0=0$; however, extending it to the general case of $x_0\in (-1,1)$ is straightforward. %

We start by defining
\begin{equation}
\mathbf T\left(  z \right)  \isdef 2^{n\sigma_{3}} \mathbf Y\left(  z \right)  \varphi\left(
z\right)  ^{-n\sigma_{3}}, \label{sol T}%
\end{equation}%
where
\begin{equation}\label{phi}
    \varphi\left(  z\right)  =z+\sqrt{z^{2}-1}%
\end{equation}%
is the conformal map from $\mathbb C \setminus [-1,1]$ onto
the exterior of the unit circle, with the branch of $\sqrt{z^2-1}$ analytic in $\mathbb C \setminus [-1,1]$
and that behaves like $z$ as $z \to \infty$.

In order to perform the second transformation we need to extend the definition of the weight of orthogonality to a neighborhood of the interval $[-1,1]$. 
By assumptions, $h$ is a holomorphic function in a neighborhood $U$ of $[-1,1]$, and positive on this interval. For any Jordan arc $\Gamma$,  intersecting $[-1,1]$ transversally at $x_{0}$ and dividing $U$ into two connected components, we denote by $\Sigma_5$ its intersection with the upper half plane, and by $\Sigma_6$ its intersection with the lower half plane, oriented as shown in Figure \ref{fig:Sigma5-6}. Contours $ \Sigma_5 \cup \Sigma_6 \cup \R$ divide $U$ into four open domains (``quadrants''), that we denote by $Q_{\pm}^{L,R}$ as depicted. Finally, let $Q^{L}$ (resp., $Q^{R}$) be the connected component of $U\setminus \Gamma $ 
containing $-1$ (resp., $+1$).

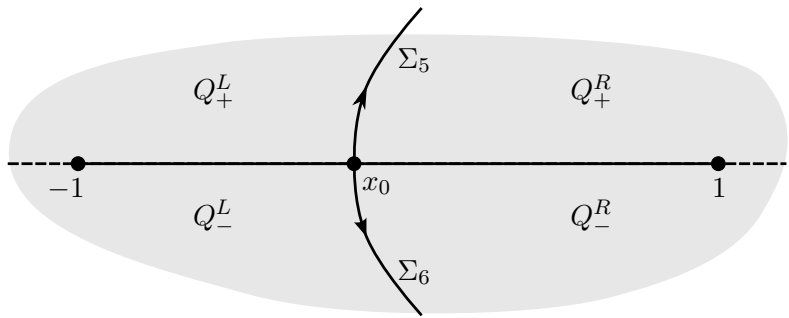
\begin{figure}[htb]
\centering \begin{overpic}[scale=1.3]%
{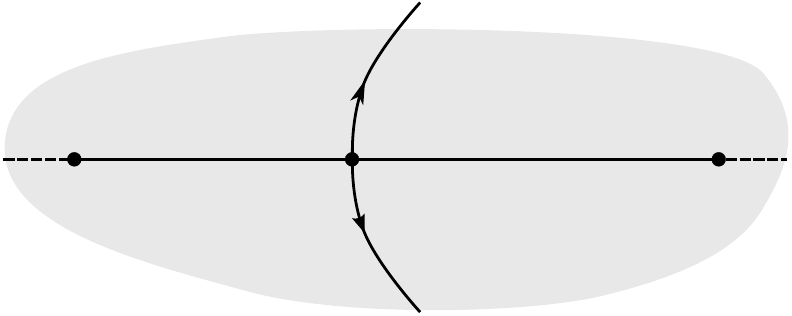}%
     \put(5.5,15.5){$\small -1 $}
\put(90,15.5){$\small 1$} 
  \put(45.5,16.5){$\small x_0$} 
  \put(24,12){$\small Q_{-}^L$}
    \put(24,28){$\small Q_{+}^L$}
      \put(72,28){$\small Q_{+}^R$}
    \put(72,12){$\small Q_{-}^R$}
    \put(50,5){$\small \Sigma_{6}$}
    \put(50,32){$\small \Sigma_{5}$}
\end{overpic}
\caption{Division of the neighborhood of $[-1,1]$  in four regions.}
\label{fig:Sigma5-6}
\end{figure}
%

Now we set
\begin{equation}\label{defwanalytic}
    w(z)\isdef h(z)\, \left(  1-z\right)  ^{\alpha}\left(  1+z\right)  ^{\beta} \times \begin{cases}
\left(  x_{0}-z\right)  ^{\gamma}, & \ z\in Q^L \setminus \left(-\infty, -1\right],\\
\left(  z-x_{0}\right)  ^{\gamma}, & \ z\in Q^R \setminus \left[1, +\infty\right),%
\end{cases}
\end{equation}
where the principal branches of the power functions are taken. In this way, $w$ is defined and holomorphic in $\widetilde U\isdef U\setminus \left( \left(-\infty, -1\right] \cup \left[1, +\infty \right) \cup \Gamma\right)$, and $w(x)>0$ for $x \in (-1,1)\setminus x_0$. %
Setting
\[
\Xi_{c}\left(  z\right)  = 
\begin{cases}
1 &   z\in Q^L\\
c^2 &   z\in Q^R,
\end{cases}
\]
we extend also
\begin{equation} \label{defwcanalytic}
w_{c,\gamma}(z) \isdef w(z)\, \Xi_{c}(z),
\end{equation}
to a holomorphic function in $\widetilde U$.

We describe now the next transformation consisting in opening of lenses or contour deformation.

\begin{figure}[htb]
\centering \begin{overpic}[scale=1.5]%
{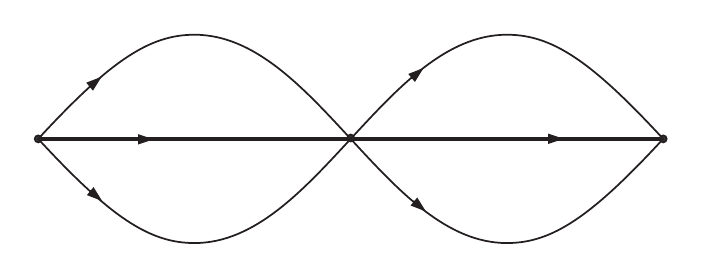}%
     \put(0,21){$\small -1 $}
\put(96,21){$\small 1$}
  \put(48.5,22.5){$\small x_0$}
  \put(19,35){$\small \Sigma_{1}$}
    \put(19,3){$\small \Sigma_{2}$}
      \put(78,35){$\small \Sigma_{3}$}
    \put(78,3){$\small \Sigma_{4}$}
    \put(59,14){\small inner lower dom.}
    \put(59,24){\small inner upper dom.}
    \put(41,33){\small outer dom.}
\end{overpic}
\caption{First lens opening.}
\label{fig:lenses1}
\end{figure}
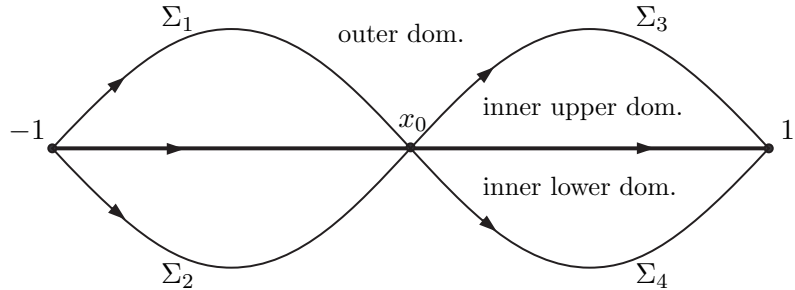

We build the four contours $\Sigma_i$ lying in $\widetilde U$ (except for their end points) such that $\Sigma_{1}$ and $\Sigma_{3}$ are in the upper half plane, and $\Sigma_{1}$ and $\Sigma_{2}$ are in the left half plane, and oriented ``from $-1$ to $1$'' but now through $x_0$ (see Fig.~\ref{fig:lenses1}).   This construction defines three domains: the inner upper domain, bounded by $[-1,1]$ and the curves $\Sigma_{1}$ and $\Sigma_{3}$; the inner lower domain, bounded by $[-1,1]$ and the curves $\Sigma_{2}$ and $\Sigma_{4}$, and finally the outer domain, bounded by curves $\Sigma_i$ and containing the infinity.
Denote $\Sigma \isdef [-1,1]\cup \bigcup_{k=1}^4\Sigma_k$.

Using the matrix $\mathbf T$ from \eqref{sol T} we define 
\begin{equation}
\mathbf S\left(  z\right)  \isdef
\begin{cases}
\mathbf T\left(  z\right),  & \text{for $z$ in the outer domain, }\\
\mathbf T(z)\, \begin{pmatrix}
1 & 0\\
-\frac{1}{w_{c,\gamma}(z)}\varphi^{-2n}(z) & 1
\end{pmatrix},
  & \text{for $z$ in the inner upper domain, }\\[3mm]
\mathbf T(z)\, \begin{pmatrix}
1 & 0\\
\frac{1}{w_{c,\gamma}(z)}\varphi^{-2n}(z) & 1
\end{pmatrix},
   & \text{for $z$ in the inner lower domain.}%
\end{cases}
  \label{sol S}%
\end{equation}
Then $\mathbf S$ is a solution of a new RH problem, now with jumps on $\Sigma$, that are easy to compute explicitly. The uniqueness is guaranteed if we impose the additional local requirement:  as $z\rightarrow x_0$,  $z\in\mathbb{C}\backslash \Sigma$,
\begin{itemize}
\item for $-1<\gamma<0$, \, $%
\mathbf S\left(  z\right)  =\mathcal O
\begin{pmatrix}
1 & \left\vert z-x_0\right\vert ^{\gamma}\\
1 & \left\vert z-x_0\right\vert ^{\gamma}%
\end{pmatrix}
  \text{, \ \ \ as }z\rightarrow x_0 $,%
\item for $\gamma=0$:%
\[
\mathbf S(z)=
\begin{cases}
\mathcal O
\begin{pmatrix}
1 & \log\left\vert z-x_0 \right\vert \\
1 & \log\left\vert z-x_0 \right\vert
\end{pmatrix},
 & \text{as }z\rightarrow x_0\text{ from the outer domain,}\\
\mathcal O
\begin{pmatrix}
\log\left\vert z-x_0 \right\vert  & \log\left\vert z-x_0 \right\vert \\
\log\left\vert z-x_0 \right\vert  & \log\left\vert z-x_0 \right\vert
\end{pmatrix},
  & \text{as }z\rightarrow x_0\text{ from the inner domains,}%
\end{cases}
\]
\item for $\gamma>0$:%
\[
\mathbf S(z)=
\begin{cases} \mathcal O
\begin{pmatrix}
1 & 1\\
1 & 1
\end{pmatrix},
 & \text{as }z\rightarrow x_0, \text{ from the outer domain,}\\
\mathcal O
\begin{pmatrix}
\left\vert z-x_0 \right\vert ^{-\gamma} & 1\\
\left\vert z-x_0 \right\vert ^{-\gamma} & 1
\end{pmatrix},
  & \text{as }z\rightarrow x_0, \text{ from the inner domain.}%
\end{cases}
\]
\end{itemize}

\subsection{Outer parametrix}

In the next step, which is also standard, we build the so-called outer parametrix for the RH problem for $\mathbf S$ in terms of the Szeg\H{o} function $D(\cdot, w_{c,\gamma})$ corresponding to the weight $w_{c,\gamma}$. Namely, we construct the $2\times 2$ matrix $N$ that satisfies
\begin{enumerate}
\item[(N1)] $\mathbf N$ is analytic in $\mathbb{C}\backslash\left[  -1,1\right]$;

\item[(N2)] $
\mathbf N_{+}(x)=\mathbf N_{-}(x)
\begin{pmatrix}
0 & w_{c,\gamma}\left(  x\right) \\
-w_{c,\gamma}\left(  x\right)  ^{-1} & 0
\end{pmatrix}$,\quad $x\in\left(
-1,x_0\right)  \cup (x_0,1) $;

\item[(N3)] 
$\mathbf N(z)=\mathbf {I}+\OO\left(  1/z\right)$, as $z\rightarrow\infty$.
\end{enumerate}%
The solution is given by
\begin{equation}
\mathbf N\left(  z\right)  \isdef D_{\infty}^{\sigma_{3}} \mathbf A(z)
D\left(  z, w_{c,\gamma}\right)  ^{-\sigma_{3}}, \label{sol-N}%
\end{equation}
and we will describe the three factor appearing in the r.h.s. of \eqref{sol-N}. Matrix $\mathbf A$ is
\begin{equation}
\label{defmatrixA}
\begin{split}
\mathbf A(z)\isdef \begin{pmatrix}
A_{11} &  A_{12}\\
-A_{12} & A_{11} \end{pmatrix}, \quad%
A_{11}(z) = \frac{\varphi\left(
z\right)  ^{1/2}}{\sqrt{2}\left(  z^{2}-1\right)  ^{1/4}}, \\
A_{12}(z) =\frac{i\varphi\left(
z\right)  ^{-1/2}}{\sqrt{2}\left(  z^{2}-1\right)  ^{1/4}}=\frac{i}{\varphi(z)}\, A_{11}(z) ,
\end{split}
\end{equation}
with the main branches of the roots, in such a way that $A_{11}$ is analytic in
$\mathbb{C}\backslash\left[  -1,1\right]  $ with $A_{11} \left(  z\right)\rightarrow1$, and $A_{12} \left(  z\right)\rightarrow 0$, as $z\rightarrow\infty$. %
The Szeg\H{o} function $D(\cdot, w_{c,\gamma})$ for $w_{c,\gamma}$ is given by
\begin{equation}\label{SzegoTotal}
    D(z, w_{c,\gamma}) = D(z, h)D(z, w_{1,\gamma})D(z, \Xi_c)\,,
\end{equation}%
where
\begin{equation}\label{SzegoPartial1}
D(z, h)= \exp\left( \sqrt{1-z^2}\,  \mathcal C \left( \frac {\log h(t)}{\sqrt{1-t^2}} \right) (z) \right), \quad
    D(z, w_{1,\gamma}) =\dfrac{\left(  z-1\right)
^{ \alpha/2 }\left(  z+1\right)  ^{\beta/2} \left(  z-x_0\right)  ^{\gamma/2}}{
\varphi^{(\alpha+\beta+\gamma )/2}(z)   },
\end{equation}%
and%
\begin{equation}\label{SzegoPartial2}
D(z, \Xi_c) = c \, \exp\left(  -\lambda \log\left( \dfrac{1-zx_0-i\sqrt{z^{2}-1}}{z-x_0}\right)\right)  ,
\end{equation}%
with%
\begin{equation}\label{def_lambda}
 \lambda\isdef   i\dfrac{\log c}{\pi},
\end{equation}%
where we take the main branches of $\left(  z-1\right)^{ \alpha/2 }$, $\left(  z+1\right)  ^{\beta/2}$, $(z-x_{0})^{\gamma/2}$ 
and $\sqrt{z^{2}-1}$ that are positive for $z>1$, as well as the main branch of the logarithm (see \cite[section 2.3]{FMS} for a detailed computation). Finally,%
\begin{equation}
D_{\infty}\isdef D(\infty, w_{c,\gamma})=\sqrt{c} \, D \left(  \infty, h \right)    2
^{-(\alpha+\beta+\gamma)/2} \, e^{i\lambda \arcsin x_0} \, >0. \label{Doo}%
\end{equation}%
Some of the properties of this function are summarized in the following lemma:

\begin{lemma}
\label{Lem-D+} The Szeg\H{o} function $D(\cdot, w)$ for the weight $w$ defined in \eqref{defwanalytic} exhibits the following boundary behavior:
\begin{equation}\label{boundaryValuesFinal}
 \lim_{\stackrel{z\to x\in  (-1, 1), }{\pm \Im z>0}} D(z, w) = \sqrt{w(x)}\,   \exp\left(  \pm i \widehat{ \Phi}(x) \right),
\end{equation}%
with
\begin{align}\label{def_PhinoC}
\Phi\left(  x\right)   &  =\frac{\pi \alpha }{2}-\dfrac { \alpha+\beta +\gamma}{2} \arccos x-\frac{\sqrt{1-x^{2}}}{2\pi}\fint_{-1}%
^{1}\frac{\log h\left(  t\right)  }{\sqrt{1-t^{2}}}\frac{dt}{t-x},\\
\widehat{\Phi}\left(  x\right)   &  =\left\{
\begin{array}
[c]{ll}%
\Phi\left(  x\right)  +\frac{\pi \gamma}{2}, & -1<x<x_{0}\\
\Phi\left(  x\right), & x_{0}<x<1
\end{array}
\right.
\end{align}%
Furthermore, for the step function $\Xi_c$,
\begin{equation*}
 \lim_{\stackrel{z\to x\in  (-1, x_0) \cup (x_0,1), }{\pm \Im z>0}} D(z, \Xi_c) = \sqrt{\Xi_c(x)}\,   \exp\left(  \mp  i\, \dfrac{\log c}{\pi} \,  \log\left| \dfrac{1-x_0x+ \sqrt{\left(1-x^{2}\right)\left(1-x_0^{2}\right)}}{x-x_0}\right| \right),
\end{equation*}%
and
\begin{equation}
D\left(  z, \Xi_c \right)  =c^{1\pm \frac{i}{\pi}\, \log\left(\frac{z-x_0}{2(1-x_0^2)} \right)}   \left[  1+o\left(
1\right)  \right],  \text{ as }z\rightarrow x_0\text{, }  \pm \Im %
z>0. \label{D-0}%
\end{equation}
\end{lemma}

The proof of this lemma is similar to \cite[Lemma 7]{FMS}, up to the difference that our jump here takes place at a generic point $x_0$, and that $w(z)$ (see \eqref{defwanalytic}) has an extra factor which makes $w(z)$ non-analytic across $\Sigma_5\cup\Sigma_6$.

The main purpose for constructing $\NN$ is that it solves the ``stripped" RH problem, obtained from the RH problem for $\mathbf S$ by ignoring all jumps asymptotically close to identity. Unfortunately, this property of $\NN$ is not uniform on the whole plane: the jumps of $\NN$ and $\mathbf S$ are no longer close in the neighborhoods of $\pm1$ and $x_0$. The analysis at these points requires a separate treatment, called local analysis, that we perform next. The outline of this analysis at each point is the following: take a small disc centered at the point and build a matrix-valued function (local parametrix) that:
\begin{enumerate}
\item[(i)] matches exactly the jumps of $\mathbf S$ within the disc, and
\item[(ii)] coincides with $\NN$ on the boundary of the disc, at least to an order $o(1)$, $n\rightarrow\infty$.
\end{enumerate}

\subsection{Local parametrix}

We fix a $\delta>0$ small enough such that discs $U_{\zeta}\isdef\left\{z\in \C: |x-\zeta|<\delta\right\},\, \zeta\in\left\{-1, x_0, 1 \right\}$ are mutually disjoint and lie in the domain of analyticity of the function $h$.
We skip the details of construction of the local parametrices $\mathbf{P}_{\pm 1}$ at $z=\pm1$ and refer the reader to \cite{MR2087231}.

For the local parametrix at the jump we need to build a $2\times 2$ matrix-valued function $\mathbf P_{x_0} \isdef \mathbf P_{0}$ in $U_{x_0}\setminus \Sigma$ that satisfies the following conditions:
\begin{enumerate}
\item[(P$_0$1)] $\mathbf P_{0}$ is holomorphic in $U_{x_0}\backslash\Sigma$ and continuous up to the boundary.

\item[(P$_0$2)] $\mathbf P_{0}$ satisfies the following jump relations:
\begin{align*}
\mathbf P_{0+}(z)  &  =\mathbf P_{0-}(z)\begin{pmatrix}
1 & 0\\
\frac{1}{w_{c,\gamma}\left(  z\right)  }\, \varphi\left(  z\right) ^{-2n} & 1
\end{pmatrix}, \quad \text{for }z\in U_{x_0}\cap\left( \bigcup_{i=1}^4\Sigma_i\right) \setminus \{ x_0 \} ;\\
\mathbf P_{0+}(x)  &  =\mathbf P_{0-}(x)\begin{pmatrix}
0 & w_{c,\gamma} \left(  z\right) \\
-\frac{1}{w_{c,\gamma}\left(  z\right)  } & 0
\end{pmatrix}, \quad \text{for }z\in U_{x_0}\cap\left( (-1,x_0)\cup (x_0,1) \right) .
\end{align*}

\item[(P$_0$3)] $\mathbf {P}_{0}(z)\mathbf {N}^{-1}\left(  z\right)  =\mathbf {I}+\OO\left(  1 / n \right)$, as $n\rightarrow\infty$,
uniformly for $z\in\partial U_{x_0}\backslash\Sigma\text{.}$%

\item[(P$_0$4)] $\mathbf P_{0}$ has the same behavior than $\mathbf S$ as $z\rightarrow x_{0} $,  $z\in
U_{x_0}\backslash\Sigma$.
\end{enumerate}

Following a standard procedure, we obtain the solution of this problem in two steps, getting first a matrix $\mathbf P^{\left(  1\right)  }$ that satisfies conditions (P$_0$1, P$_0$2, P$_0$4). After that, using an additional freedom in the construction, we take care of the matching condition (P$_0$3).

We define an auxiliary function $W$, holomorphic in $\widetilde U \setminus \R$. In the next formula we understand by $\left(  h\left(  z\right)  \left(  1-z\right)  ^{\alpha}\left(  1+z\right) ^{\beta}\left(  z-x_{0}\right)  ^{\gamma}c\right)  ^{1/2}$ the holomorphic branch of this function in $U\setminus ((-\infty, x_{0}]\cup [1,+\infty))$, positive on $(x_{0},1)$. Analogously, $\left(  h\left(  z\right)  \left(  1-z\right)  ^{\alpha}\left(  1+z\right)
^{\beta}\left(  x_{0}-z\right)  ^{\gamma}c\right)  ^{1/2}$ stands for  the holomorphic branch in $U\setminus ((-\infty, -1]\cup [x_{0},+\infty))$, positive on $(-1,x_{0})$. With this convention we set
\begin{equation}\label{Def-W}
W\left(  z\right)  =\begin{cases}
\left(  h\left(  z\right)  \left(  1-z\right)  ^{\alpha}\left(  1+z\right)
^{\beta}\left(  z-x_{0}\right)  ^{\gamma}c\right)  ^{1/2}, &
 z\in Q^{L}_{+}\cup Q^{L}_{-},\\
\left(  h\left(  z\right)  \left(  1-z\right)  ^{\alpha}\left(  1+z\right)
^{\beta}\left(  x_{0}-z\right)  ^{\gamma}c\right)  ^{1/2}, &
 z\in Q^{R}_{+}\cup Q^{R}_{-}.%
\end{cases}
\end{equation}
It is easy to see from \eqref{defwanalytic}--\eqref{defwcanalytic} that 
\begin{equation*}
W^{2}\left(
z\right)  = 
\begin{cases}
w_{c,\gamma}\left(  z\right)  e^{-\gamma\pi i} c^{-1}, &  z\in Q_+^R,\\
w_{c,\gamma}\left(  z\right)  e^{\gamma\pi i} c, & z\in Q_+^L,\\
w_{c,\gamma}\left(  z\right)  e^{-\gamma\pi i} c, & z\in Q_-^L,,\\
w_{c,\gamma}\left(  z\right)  e^{\gamma\pi i} c^{-1}, & z\in Q_-^R.%
\end{cases}
\end{equation*}
This shows that $W$ satisfies the following jump relations:
\begin{equation*}
W_{+}\left(  x\right)   = 
\begin{cases}
W_{-}\left(  x\right)  e^{i\gamma\pi}, &  -1<x<x_{0},\\
W_{-}\left(  x\right)  e^{-i\gamma\pi}, &  x_{0}<x<1,
\end{cases}
\end{equation*}
and
\begin{equation}\label{W2}
W_{+}\left(  z\right)    =e^{i \gamma\pi/ 2} \, W_{-}\left(  z\right) ,\quad  z\in\Sigma_{5}\cup\Sigma_{6}.
\end{equation}
Moreover, 
\begin{equation}\label{W4}
W_{+}\left(  x\right)
= 
\begin{cases}
\sqrt{w_{c,\gamma}\left(  x\right) c } \, e^{i\frac{\gamma\pi}{2}} = \sqrt{w_{1,\gamma}\left(  x\right) c }\, e^{i\frac{\gamma\pi}{2}}, &  -1<x<x_{0},\\
\sqrt{w_{c,\gamma}\left(  x\right) c^{-1} } \, e^{-i\frac{\gamma\pi}{2}} = \sqrt{w_{1,\gamma}\left(  x\right) c }\, e^{-i\frac{\gamma\pi}{2}}, &  x_{0}<x<1,
\end{cases}
\end{equation}
and
\begin{equation}\label{W1}
W_{+}\left(  x\right)  W_{-}\left(  x\right)  = 
\begin{cases}
w_{c,\gamma}\left(  x\right)  c, &  -1<x<x_{0},\\
w_{c,\gamma}\left(  x\right)  c^{-1}, &  x_{0}<x<1.
\end{cases}
\end{equation}

We construct the matrix function $\mathbf P_{0}$ in the following form:
\begin{equation}
\mathbf P_{0}\left(  z\right)  =\mathbf E_{n}\left(  z\right) \mathbf   P^{\left(  1\right)  }\left(
z\right)  W\left(  z\right)  ^{-\sigma_{3}}\varphi\left(  z\right)
^{-n\sigma_{3}}, \label{Def-P}%
\end{equation}
where $\mathbf E_n$ is an analytic matrix-valued function in $U_{x_0}$ (to be determined). Matrix $\mathbf P^{\left(  1\right)  }$ is analytic in $U_{x_0}\setminus \Sigma$. Denote by%
\begin{align}\label{Psi_k}
J_1 = \begin{pmatrix}
0 & c \\
-1/c & 0
\end{pmatrix},\quad
J_2 = \begin{pmatrix}
1 & 0 \\
e^{-\gamma \pi i}c^{-1} & 1
\end{pmatrix},\quad
J_3 = J_7 = \begin{pmatrix}
e^{i\gamma \pi /2} & 0 \\
0 & e^{-i\gamma \pi /2}
\end{pmatrix},\\ \label{Psi_k-2}
J_4 = \begin{pmatrix}
1 & 0 \\
e^{\gamma \pi i}c & 1
\end{pmatrix},\quad
J_5 = \begin{pmatrix}
0 & 1/c \\
-c & 0
\end{pmatrix},\quad
J_6 = \begin{pmatrix}
1 & 0 \\
e^{-\gamma \pi i}c & 1
\end{pmatrix}.\quad \quad \quad
\end{align}%
Using the properties of $W$ and $\varphi$ it is easy to show that
\begin{equation}\label{jumpsForP1_1}
\mathbf P_{+}^{\left(  1\right)  }\left(  x\right)  =\mathbf P_{-}^{\left(  1\right)  }\left(
x\right)
\begin{cases}
J_5, & x\in (x_0-\delta, x_0), \\[0mm]
J_1, & x\in (x_0, x_0+ \delta),
\end{cases}
\end{equation}%
and
\begin{equation}\label{jumpsForP1_2}
\mathbf P_{+}^{\left(  1\right)  }\left(  z\right)  =\mathbf P_{-}^{\left(  1\right)  }\left(
z\right) \begin{cases}
J_4, & z\in \Sigma_{1} \cap U_{x_0}\setminus \{x_0\}, \\[0mm]
J_6, & z\in  \Sigma_{2}\cap U_{x_0}\setminus \{x_0\}, \\[0mm]
J_2,&z\in \Sigma_{3} \cap U_{x_0}\setminus \{x_0\}, \\[0mm]
J_8 ,&z\in \Sigma_{4} \cap U_{x_0}\setminus \{x_0\},
\end{cases}
\end{equation}%
and, as $W$ has a jump on $\Sigma_5\cup\Sigma_6$, by \eqref{W2}, we have two additional jumps on $\Sigma_5\cup\Sigma_6$: %
\begin{equation}\label{jumpsForP1_2-2}
\mathbf P_{+}^{\left(  1\right)  }\left(  z\right)  =\mathbf P_{-}^{\left(  1\right)  }\left(
z\right)
\begin{cases}
J_3,&z\in \Sigma_{5} \cap U_{x_0}\setminus \{x_0\}, \\[0mm]
J_7,&z\in \Sigma_{6} \cap U_{x_0}\setminus \{x_0\}.
\end{cases}
\end{equation}%

Taking into account that $W\left(  z\right)  =\OO\left( |z-x_0|^{\gamma/2}\right)  $ and $\varphi\left(  z\right)
=\OO\left(  1\right)  $ as $z\rightarrow x_0$, we conclude also from (P$_0$4) that
 $\mathbf P^{\left(  1\right)  }$ has the following behavior at $x_0$: as $z\rightarrow x_0$,  $z\in\mathbb{C}\backslash \left(\Sigma \cup \Sigma_5 \cup \Sigma_6\right)$,
\begin{itemize}
\item for $\gamma<0$:
\begin{equation}\label{jumpsForP1_3}
\mathbf{P}^{\left(  1\right)  }(z)=
\OO\left(
\begin{array}
[c]{cc}%
\left\vert z-x_{0}\right\vert ^{\gamma/2} & \left\vert z-x_{0}\right\vert
^{\gamma/2}\\
\left\vert z-x_{0}\right\vert ^{\gamma/2} & \left\vert z-x_{0}\right\vert
^{\gamma/2}%
\end{array}
\right) ,%
\end{equation}

\item for $\gamma=0$%
\begin{equation}\label{jumpsForP1_3-2}
\mathbf{P}^{\left(  1\right)  }(z)=\left\{
\begin{array}
[c]{ll}%
\OO\left(
\begin{array}
[c]{cc}%
\log\left\vert z-x_{0}\right\vert  & \log\left\vert z-x_{0}\right\vert \\
\log\left\vert z-x_{0}\right\vert  & \log\left\vert z-x_{0}\right\vert
\end{array}
\right),  & \text{ from inside the lens,}\\
\OO\left(
\begin{array}
[c]{ll}%
1 & \log\left\vert z-x_{0}\right\vert \\
1 & \log\left\vert z-x_{0}\right\vert
\end{array}
\right),  & \text{ from outside the lens,}%
\end{array}
\right.
\end{equation}

\item for $\gamma>0:$%
\begin{equation}\label{jumpsForP1_3-3}
\mathbf{P}^{(1)}(z)=\left\{
\begin{array}
[c]{ll}%
\OO\left(
\begin{array}
[c]{cc}%
\left\vert z-x_{0}\right\vert ^{\gamma/2} & \left\vert z-x_{0}\right\vert
^{-\gamma/2}\\
\left\vert z-x_{0}\right\vert ^{\gamma/2} & \left\vert z-x_{0}\right\vert
^{-\gamma/2}%
\end{array}
\right),  & \text{ from outside the lens,}\\
\OO\left(
\begin{array}
[c]{cc}%
\left\vert z-x_{0}\right\vert ^{-\gamma/2} & \left\vert z-x_{0}\right\vert
^{-\gamma/2}\\
\left\vert z-x_{0}\right\vert ^{-\gamma/2} & \left\vert z-x_{0}\right\vert
^{-\gamma/2}%
\end{array}
\right),  & \text{ from inside the lens.}%
\end{array}
\right.
\end{equation}

\end{itemize}

In order to construct $\mathbf P^{(1)}$ we solve first an auxiliary RH problem on a set $\Gamma\isdef \bigcup_{j=1}^8 \Gamma_j$ of unbounded oriented straight lines converging at the origin, like in Fig.~\ref{fig:psicontours}.
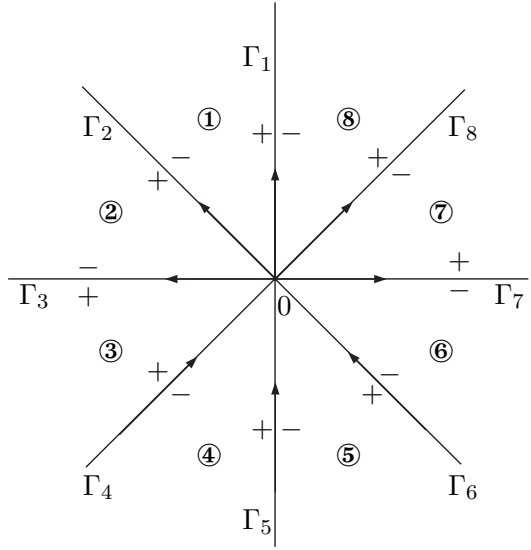
\begin{figure}[htb]
\centering \begin{overpic}[scale=0.7]%
{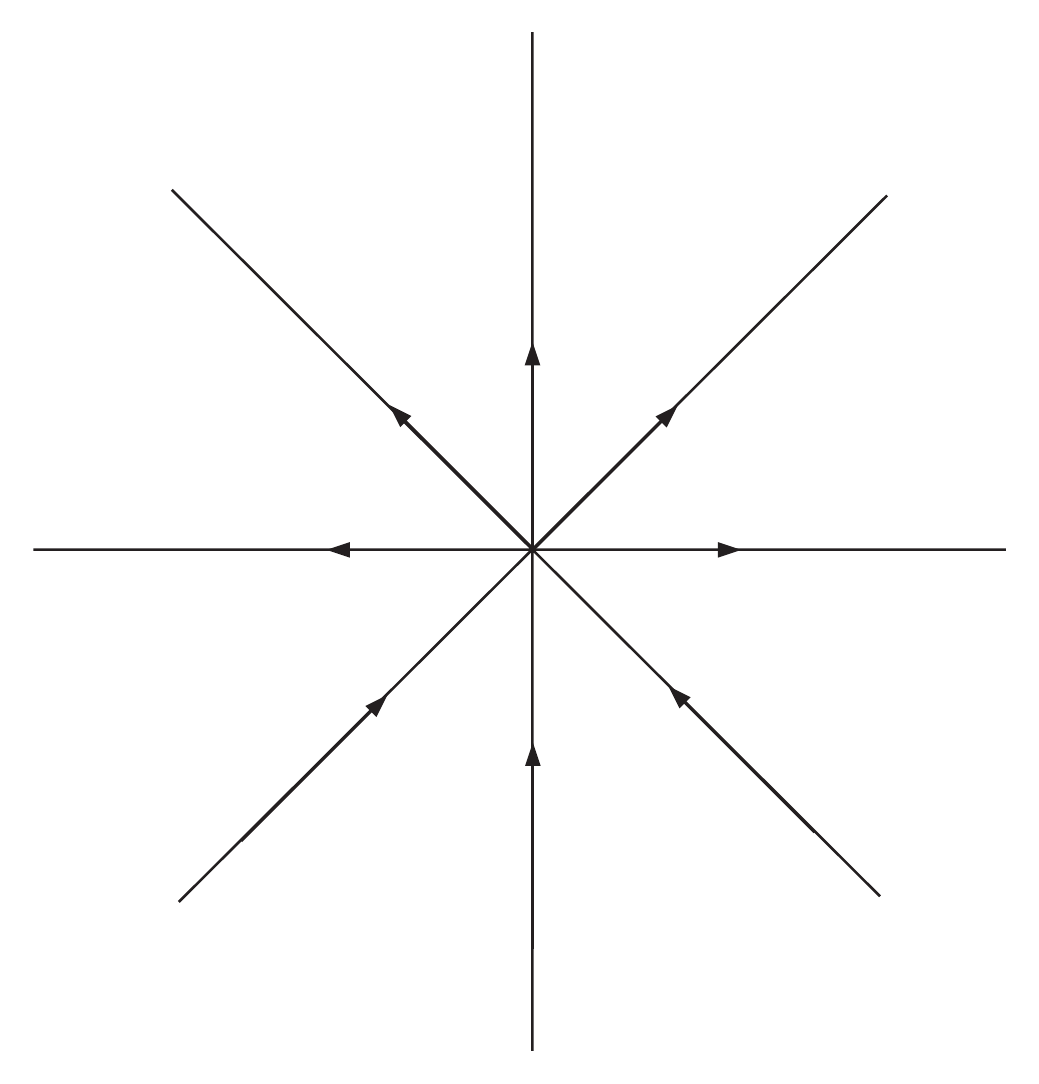}%
    \put(34,75){  \large \ding{172}  }
 \put(17,59){  \large  \ding{173}  }
 \put(34,17){  \large  \ding{175}  }
    \put(58,75){  \large \ding{179}  }
 \put(74,35){  \large  \ding{177}  }
 \put(58,17){  \large  \ding{176}  }
 \put(74,59){  \large \ding{178}  }
 \put(17,35){  \large \ding{174}  }
\put(15,45){$\small +$}
\put(15,50){$\small -$}
\put(79,51){$\small +$}
\put(79,46){$\small -$}
\put(27,32){$\small +$}
\put(31,28){$\small -$}
\put(65,69){$\small +$}
\put(69,66){$\small -$}
\put(45,73){$\small +$}
\put(50,73){$\small -$}
\put(45,22){$\small +$}
\put(50,22){$\small -$}
\put(27,65){$\small +$}
\put(31,69){$\small -$}
\put(64,28){$\small +$}
\put(67,31.5){$\small -$}
\put(49.5,43){$\small 0$}
\put(43.5,86){$\Gamma_1$}
\put(43.5,6){$\Gamma_5$}
\put(16,12){$\Gamma_4$}
\put(5,45){$\Gamma_3$}
\put(16,74){$\Gamma_2$}
\put(79,12){$\Gamma_6$}
\put(87,45){$\Gamma_7$}
\put(79,74){$\Gamma_8$}
\end{overpic}
\caption{Auxiliary contours $\Gamma$.}
\label{fig:psicontours}
\end{figure}
More precisely,
\begin{align*}
 \Gamma_1 &=\left\{t\,e^{i\pi /2}:\, t>0 \right\}, \, \Gamma_2 =\left\{t\,e^{3 i\pi /4}:\, t>0 \right\}, \, \Gamma_3 =\left\{-t:\, t>0 \right\}, \, \Gamma_4 =\left\{t\,e^{5 i\pi /4}:\, t>0 \right\},\\
 \Gamma_5 &=\left\{t\,e^{3 i\pi /2}:\, t>0 \right\}, \, \Gamma_6 =\left\{t\,e^{- i\pi t/4}:\, t>0 \right\}, \, \Gamma_7 =\left\{t:\, t>0 \right\}, \, \Gamma_8 =\left\{e^{  i\pi t/4}:\, t>0 \right\}.
\end{align*}
These lines split the plane into 8 sectors, enumerated anti-clockwise from \ding{172} to \ding{179} as in Fig.~\ref{fig:psicontours}.

We look for a $2\times2$ matrix valued function
$\mathbf \Psi\left(  z\right)   $, satisfying the following conditions:
\begin{enumerate}
\item[($\Psi$1)] $\mathbf \Psi$ is analytic in $\mathbb{C}\backslash\Gamma$.

\item[($\Psi$2)] for $k=1, \dots, 8$, $\mathbf \Psi$ satisfies the jump relation $\mathbf \Psi_{+}(\zeta)     =\mathbf \Psi_{-}(\zeta) J_k$ on $\Gamma_k$, with $J_k$ given by \eqref{Psi_k} and \eqref{Psi_k-2}.

\item[($\Psi$3)] the behavior of $\mathbf \Psi$ as $\zeta\rightarrow0$ is obtained from that of $\PP^{(1)}$ at $x_{0}$ by replacing $\left(z-x_0 \right)$ with $\zeta$. Now  the region ``inside lens" correspond to  $\text{\ding{172}}\cup \text{\ding{175}}\cup \text{\ding{176}}\cup \text{\ding{179}}$ and the region ``outside lens" corresponds to  $\text{\ding{173}}\cup \text{\ding{174}}\cup \text{\ding{177}}\cup \text{\ding{178}}$.%
\end{enumerate}%
We construct $\mathbf \Psi$ explicitly using the confluent hypergeometric
functions
\[
\phi\left(  a,\gamma+1;\zeta\right)  \isdef {_{1}F_{1}}\left(
a;\gamma+1;\zeta\right)  \text{ \ \ and \ }\psi\left(  a,\gamma+1;\zeta
\right)  \isdef z^{-a} {_{2}F_{0}}\left(  a,a-\gamma;-;-1/\zeta\right)
,
\]
that are solutions of the confluent hypergeometric equation $\zeta
w^{\prime\prime}+\left(  \gamma+1-\zeta\right)  w^{\prime}-aw=0$, see
\cite[formula (13.1.1)]{abramowitz/stegun:1972}. Namely, let
\begin{equation}
G\left(  a, \gamma; \zeta\right)  \isdef\zeta^{\gamma/2}\phi\left(
a,\gamma+1;\zeta\right)  e^{-\zeta/2}, \quad H\left(  a, \gamma; \zeta\right)
\isdef\zeta^{\gamma/2}\psi\left(  a,\gamma+1;\zeta\right)
e^{-\zeta/2}, \label{Def-G-H}%
\end{equation}
they form a basis of solutions of the confluent equation (see e.g. \cite[formula (13.1.35)]{abramowitz/stegun:1972})%
\begin{equation}
4\zeta^{2}w^{\prime\prime}+4\zeta w^{\prime}+\left[  -\gamma^{2}+2\zeta\left(
\gamma+1-2a\right)  -\zeta^{2}\right]  w=0. \label{eq-dif-confl}%
\end{equation}
We can relate $G$ and $H$ with the Whittaker functions: $G(a,\gamma;z)=M_{\kappa,\mu}(z)/\sqrt z$ and $H(a,\gamma;z)=W_{\kappa,\mu}(z)/\sqrt z$ with $\mu=\gamma/2$ and $\kappa=1/2+\mu-a$ (see \cite[formula (13.1.32)]{abramowitz/stegun:1972}). 

In general, $G\left(  a, \gamma; \zeta \right)  $  and $H\left(  a, \gamma; \zeta \right)  $ from (\ref{eq-dif-confl}) are multi-valued, and we take its principal branch in $-\frac{\pi
}{2}<\arg\left(  \zeta\right)  <\frac{3\pi}{2}$. For these values of $\zeta$ we define%
\[
\widehat{\mathbf \Psi}\left(  \zeta\right)  \isdef \left(
\begin{array}
[c]{cc}%
\frac{\Gamma\left(  1-\lambda+\frac{\gamma}{2}\right)  }{\Gamma\left(
\gamma+1\right)  }G\left(  \lambda+\frac{\gamma}{2},\gamma;\zeta\right)  & -H\left(
\lambda+\frac{\gamma}{2},\gamma;\zeta\right) \\
\frac{\Gamma\left(  1+\lambda+\frac{\gamma}{2}\right)  }{\Gamma\left(
\gamma+1\right)  }G\left(  1+\lambda+\frac{\gamma}{2},\gamma;\zeta\right)  &
\frac{\Gamma\left(  1+\lambda+\frac{\gamma}{2}\right)  }{\Gamma\left(
\frac{\gamma}{2}-\lambda\right)  }H\left(  1+\lambda+\frac{\gamma}{2}%
,\gamma;\zeta\right)
\end{array}
\right)  e^{\frac{\gamma\pi i}{4}\sigma_{3}}.
\]
By ($\Psi_{2}$), if we set
\begin{equation}
\mathbf \Psi\left(  \zeta\right)  \isdef \left\{
\begin{array}
[c]{ll}%
\widehat{\mathbf \Psi}\left(  \zeta\right)  J_{8}J_{1}, & \text{for }\zeta
\in\text{\ding{172}};\\
\widehat{\mathbf \Psi}\left(  \zeta\right)  J_{8}J_{1}J_{2}, & \text{for }\zeta
\in\text{\ding{173}};\\
\widehat{\mathbf \Psi}\left(  \zeta\right)  J_{8}J_{1}J_{2}J_{3}, & \text{for }%
\zeta\in\text{\ding{174}};\\
\widehat{\mathbf \Psi}\left(  \zeta\right)  J_{8}J_{1}J_{2}J_{3}J_{4}^{-1}, &
\text{for }\zeta\in\text{\ding{175}};\\
\widehat{\mathbf \Psi}\left(  \zeta\right)  J_{7}^{-1}J_{6}, & \text{for }\zeta
\in\text{\ding{176}};\\
\widehat{\mathbf \Psi}\left(  \zeta\right)  J_{7}^{-1}, & \text{for }\zeta
\in\text{\ding{177}};\\
\widehat{\mathbf \Psi}\left(  \zeta\right)  , & \text{for }\zeta\in\text{\ding{178}};\\
\widehat{\mathbf \Psi}\left(  \zeta\right)  J_{8}, & \text{for }\zeta\in
\text{\ding{179}}.
\end{array}
\right.  \label{sol Psi0}%
\end{equation}
then $\mathbf \Psi$ has the jumps across $\Gamma$ specified in ($\Psi_{2}$).
Explicitly, $\mathbf \Psi\left(  \zeta\right)  =$%
\begin{equation}
\left(
\begin{array}
[c]{cc}%
c^{-1}H\left(  \lambda+\frac{\gamma}{2},\gamma;\zeta\right) &
-\frac{\Gamma\left(  1-\lambda+\frac{\gamma}{2}\right)  }{\Gamma\left(
\frac{\gamma}{2}+\lambda\right)  }H\left(  1-\lambda+\frac{\gamma}{2},\gamma;\zeta
e^{-\pi i}\right) \\
-c^{-1}\frac{\Gamma\left(  1+\lambda+\frac{\gamma}{2}\right)  }{\Gamma\left(
\frac{\gamma}{2}-\lambda\right)  }H\left(  1+\lambda+\frac{\gamma}{2}%
,\gamma;\zeta\right)   & H\left(  \frac{\gamma}{2}-\lambda,\gamma;\zeta
e^{-\pi i}\right)
\end{array}
\right)  e^{-\frac{\gamma\pi i}{4}\sigma_{3}},~\zeta\in\text{\ding{172}},
\label{sol Psi1}%
\end{equation}%
\begin{equation}
\left(
\begin{array}
[c]{cc}%
\frac{\Gamma\left(  1-\lambda+\frac{\gamma}{2}\right)  }{\Gamma\left(
\gamma+1\right)  }G\left(  \lambda+\frac{\gamma}{2},\gamma;\zeta\right)
e^{-\frac{\gamma\pi i}{2}} & -\frac{\Gamma\left(  1-\lambda+\frac{\gamma}%
{2}\right)  }{\Gamma\left(  \frac{\gamma}{2}+\lambda\right)  }H\left(
1-\lambda+\frac{\gamma}{2},\gamma;\zeta e^{-\pi i}\right) \\
\frac{\Gamma\left(  1+\lambda+\frac{\gamma}{2}\right)  }{\Gamma\left(
\gamma+1\right)  }G\left(  1+\lambda+\frac{\gamma}{2},\gamma;\zeta\right)
e^{-\frac{\gamma\pi i}{2}} & H\left(  \frac{\gamma}{2}-\lambda,\gamma;\zeta e^{-\pi
i}\right)
\end{array}
\right)  e^{-\frac{\gamma\pi i}{4}\sigma_{3}},~\zeta\in\text{\ding{173}},
\label{sol Psi2}%
\end{equation}%
\begin{equation}
\left(
\begin{array}
[c]{cc}%
\frac{\Gamma\left(  1-\lambda+\frac{\gamma}{2}\right)  }{\Gamma\left(
\gamma+1\right)  }G\left(  \lambda+\frac{\gamma}{2},\gamma;\zeta\right)  &
-\frac{\Gamma\left(  1-\lambda+\frac{\gamma}{2}\right)  }{\Gamma\left(
\frac{\gamma}{2}+\lambda\right)  }H\left(  1-\lambda+\frac{\gamma}{2},\gamma;\zeta
e^{-\pi i}\right)  e^{-\frac{\gamma\pi i}{2}}\\
\frac{\Gamma\left(  1+\lambda+\frac{\gamma}{2}\right)  }{\Gamma\left(
\gamma+1\right)  }G\left(  1+\lambda+\frac{\gamma}{2},\gamma;\zeta\right)  & H\left(
\frac{\gamma}{2}-\lambda,\gamma;\zeta e^{-\pi i}\right)  e^{-\frac{\gamma\pi i}{2}}%
\end{array}
\right)  e^{-\frac{\gamma\pi i}{4}\sigma_{3}},~\zeta\in\text{\ding{174}},
\label{sol Psi3}%
\end{equation}
\begin{equation}
\left(
\begin{array}
[c]{cc}%
c\,H\left(  \lambda+\frac{\gamma}{2},\gamma;\zeta e^{-2\pi i}\right)
& -\frac{\Gamma\left(  1-\lambda+\frac{\gamma}{2}\right)  }{\Gamma\left(
\frac{\gamma}{2}+\lambda\right)  }H\left(  1-\lambda+\frac{\gamma}{2},\gamma;\zeta
e^{-\pi i}\right) \\
-c\frac{\Gamma\left(  1+\lambda+\frac{\gamma}{2}\right)  }{\Gamma\left(
\frac{\gamma}{2}-\lambda\right)  }H\left(  1+\lambda+\frac{\gamma}{2},\gamma;\zeta
e^{-2\pi i}\right)  & H\left(  \frac{\gamma}{2}%
-\lambda,\gamma;\zeta e^{-\pi i}\right)
\end{array}
\right)  e^{\frac{\gamma\pi i}{4}\sigma_{3}},~\zeta\in\text{\ding{175}},
\label{sol Psi4}%
\end{equation}%
\begin{equation}
\left(
\begin{array}
[c]{cc}%
-\frac{\Gamma\left(  1-\lambda+\frac{\gamma}{2}\right)  }{\Gamma\left(
\lambda+\frac{\gamma}{2}\right)  }H\left(  1-\lambda+\frac{\gamma}{2},\gamma;\zeta
e^{\pi i}\right)  e^{-\lambda\pi i} & -H\left(  \lambda+\frac{\gamma}{2}%
,\gamma;\zeta\right) \\
H\left(  \frac{\gamma}{2}-\lambda,\gamma;\zeta e^{\pi i}\right)  e^{-\lambda\pi i} &
\frac{\Gamma\left(  1+\lambda+\frac{\gamma}{2}\right)  }{\Gamma\left(
\frac{\gamma}{2}-\lambda\right)  }H\left(  1+\lambda+\frac{\gamma}{2}%
,\gamma;\zeta\right)
\end{array}
\right)  e^{-\frac{\gamma\pi i}{4}\sigma_{3}},~\zeta\in\text{\ding{176}},
\label{sol Psi5}%
\end{equation}%
\begin{equation}
\left(
\begin{array}
[c]{cc}%
\frac{\Gamma\left(  1-\lambda+\frac{\gamma}{2}\right)  }{\Gamma\left(
\gamma+1\right)  }G\left(  \lambda+\frac{\gamma}{2},\gamma;\zeta\right)  & -H\left(
\lambda+\frac{\gamma}{2},\gamma;\zeta\right) \\
\frac{\Gamma\left(  1+\lambda+\frac{\gamma}{2}\right)  }{\Gamma\left(
\gamma+1\right)  }G\left(  1+\lambda+\frac{\gamma}{2},\gamma;\zeta\right)  &
\frac{\Gamma\left(  1+\lambda+\frac{\gamma}{2}\right)  }{\Gamma\left(
\frac{\gamma}{2}-\lambda\right)  }H\left(  1+\lambda+\frac{\gamma}{2}%
,\gamma;\zeta\right)
\end{array}
\right)  e^{-\frac{\gamma\pi i}{4}\sigma_{3}},~\zeta\in\text{\ding{177}},
\label{sol Psi6}%
\end{equation}%
\begin{equation}
\left(
\begin{array}
[c]{cc}%
\frac{\Gamma\left(  1-\lambda+\frac{\gamma}{2}\right)  }{\Gamma\left(
\gamma+1\right)  }G\left(  \lambda+\frac{\gamma}{2},\gamma;\zeta\right)  & -H\left(
\lambda+\frac{\gamma}{2},\gamma;\zeta\right) \\
\frac{\Gamma\left(  1+\lambda+\frac{\gamma}{2}\right)  }{\Gamma\left(
\gamma+1\right)  }G\left(  1+\lambda+\frac{\gamma}{2},\gamma;\zeta\right)  &
\frac{\Gamma\left(  1+\lambda+\frac{\gamma}{2}\right)  }{\Gamma\left(
\frac{\gamma}{2}-\lambda\right)  }H\left(  1+\lambda+\frac{\gamma}{2}%
,\gamma;\zeta\right)
\end{array}
\right)  e^{\frac{\gamma\pi i}{4}\sigma_{3}},~\zeta\in\text{\ding{178}},
\label{sol Psi7}%
\end{equation}%
\begin{equation}
\left(
\begin{array}
[c]{cc}%
-c^{-1}\frac{\Gamma\left(  1-\lambda+\frac{\gamma}{2}\right)  }{\Gamma\left(
\frac{\gamma}{2}+\lambda\right)  }H\left(  1-\lambda+\frac{\gamma}{2},\gamma;\zeta
e^{-\pi i}\right) & -H\left(  \lambda+\frac{\gamma}{2}%
,\gamma;\zeta\right) \\
c^{-1}\,H\left(  \frac{\gamma}{2}-\lambda,\gamma;\zeta e^{-\pi i}\right) &
\frac{\Gamma\left(  1+\lambda+\frac{\gamma}{2}\right)  }{\Gamma\left(
\frac{\gamma}{2}-\lambda\right)  }H\left(  1+\lambda+\frac{\gamma}{2}%
,\gamma;\zeta\right)
\end{array}
\right)  e^{\frac{\gamma\pi i}{4}\sigma_{3}},~\zeta\in\text{\ding{179}}.
\label{sol Psi8}%
\end{equation}

\begin{proposition}
The solution of the RH problem $(\Psi1)$, $(\Psi2)$, $(\Psi3)$ is given by \eqref{sol Psi0} and $\det \mathbf \Psi (z)=1$, for $z\in\C\setminus\Gamma$.
\end{proposition}
\begin{proof}
If we take the branch cut across $\arg\zeta= -\pi/2$ oriented towards the origin (we consider $ -\pi/2 <\arg\zeta< 3\pi/4$), we have
that the matrix $\mathbf \Psi$ has on this cut the following jump (using
\eqref{def_lambda}):%
\begin{align} \label{salto}
\mathbf \Psi_{+} (\zeta) &  =\mathbf \Psi_{-}(\zeta)J_{5}, \quad \zeta\in\Gamma_5,\\
\widehat{\mathbf\Psi}_{+} (\zeta) &  =\widehat{\mathbf\Psi}_{-}(\zeta) 
\begin{pmatrix}
e^{i\pi\gamma} & -e^{-i\pi\lambda}+e^{i\pi\lambda}e^{-i\pi\gamma}\\
0 & e^{-i\pi\gamma}%
\end{pmatrix}, \quad \zeta\in\Gamma_5  \text{.}%
\end{align}
Using the following relations (see \cite[appendix: formulas (7.18), (7.30), (7.27)]{Its07b}),
\begin{equation}\label{form1}%
\phi\left(  a,b;e^{\pm2\pi i}z\right)  =\phi\left(  a,b;z\right),
\end{equation}%
\begin{equation*}
\psi\left(  a,b;e^{2\pi i}z\right)  =e^{-2i\pi a}\psi\left(  a,b;z\right)
+e^{-i\pi a}\frac{2\pi i}{\Gamma\left(  a\right)  \Gamma\left(  1+a-b\right)
}\psi\left(  b-a,1;e^{i\pi}z\right)  e^{z},%
\end{equation*}%
\begin{equation*}
\psi\left(  b-a,b;e^{i\pi}z\right)  e^{z}=\frac{-\Gamma\left(  a\right)
}{\Gamma\left(  b-a\right)  }e^{-i\pi b}\psi\left(  a,b;z\right)
+\frac{\Gamma\left(  a\right)  }{\Gamma\left(  b\right)  }e^{-i\pi\left(
b-a\right)  }\phi\left(  a,b;z\right),
\end{equation*}%
\begin{equation*}
\Gamma\left(  s\right)  \Gamma\left(  1-s\right)  =\frac{2\pi i}{e^{i\pi
s}-e^{-i\pi s}},%
\end{equation*}
and, combining the last three formulas we obtain:%
\begin{equation} \label{form2}
\psi\left(  a,b;e^{2\pi i}z\right)  =\psi\left(  a,b;z\right)  e^{-2\pi
ib}+\phi\left(  a,b;z\right)  \frac{2\pi i}{\Gamma\left(  1+a-b\right)
\Gamma\left(  b\right)  }e^{-\pi ib}.%
\end{equation}
Set
\begin{align*}
\widehat{\mathbf\Psi}_{11}\left(  \zeta\right) & =\frac{\Gamma\left(  1-\lambda
+\frac{\gamma}{2}\right)  }{\Gamma\left(  \gamma+1\right)  }\zeta^{\gamma/2}\phi\left(  \lambda+\frac{\gamma}{2},\gamma+1;\zeta\right)  e^{-\zeta/2}%
e^{i\pi\gamma/4},
\\
\widehat{\mathbf\Psi}_{12} & =-\zeta^{\gamma /2}\psi\left(  \lambda+\frac{\gamma}%
{2},\gamma+1;\zeta\right)  e^{-\zeta/2}e^{-i\pi\gamma /4}.%
\end{align*}
Then from \eqref{form1} and \eqref{form2} if follows that for $\zeta\in\Gamma_{5}$,
\begin{align*}
\left(  \widehat{\mathbf\Psi}_{11}\right)  _{+}\left(  \zeta\right)   &  =\left(  e^{2\pi
i}\zeta\right)  ^{\gamma/2}\phi\left(  \lambda+\frac{\gamma}{2}%
,\gamma+1;e^{2\pi i}\zeta\right)  e^{-\zeta/2}e^{i\pi\gamma/4}\tfrac
{\Gamma\left(  1-\lambda+\frac{\gamma}{2}\right)  }{\Gamma\left(
\gamma+1\right)  }\\
&  =e^{i\pi\gamma}\left(  \widehat{\mathbf\Psi}_{11}\right)  _{-}\left(  \zeta\right)  ,
\end{align*}%
and
\begin{align*}
\left(  \widehat{\mathbf\Psi}_{12}\right)  _{+}\left(  \zeta\right)   &  =-\left(
e^{2\pi i}\zeta\right)  ^{\gamma/2}\psi\left(  \lambda+\frac{\gamma}%
{2},\gamma+1;e^{2\pi i}\zeta\right)  e^{-\zeta/2}e^{-i\pi\gamma/4} \\
&  =\tfrac{2\pi ie^{-\pi i}e^{-i\pi\gamma/4}\Gamma\left(
\gamma+1\right)  }{\Gamma\left(  \lambda-\frac{\gamma}{2}\right)
\Gamma\left(  \gamma+1\right)  \Gamma\left(  1-\lambda+\frac{\gamma}%
{2}\right)  e^{i\pi\gamma/4}}\left(  \widehat{\mathbf\Psi}_{11}\right)
_{-}\left(  \zeta\right)  +e^{\pi i\left(  \gamma-2\gamma-2\right)  }\left(
\widehat{\mathbf\Psi}_{12}\right)  _{-}\left(  \zeta\right) \\
&  =\left[  -e^{-i\pi\lambda}+e^{i\pi\lambda}e^{-i\pi\gamma}\right]  \left(
\widehat{\mathbf\Psi}_{11}\right)  _{-}\left(  \zeta\right)  +e^{-i\pi\gamma}\left(
\widehat{\mathbf\Psi}_{12}\right)  _{-}\left(  \zeta\right)  ,
\end{align*}
in accordance with  \eqref{salto}.
Analogously, we can satisfy the second row of \eqref{salto} if we take%
\[
\widehat{\mathbf\Psi}_{21}=\frac{\Gamma\left(  1+\lambda+\frac{\gamma}{2}\right)
}{\Gamma\left(  \gamma+1\right)  }\zeta^{\gamma/2}\phi\left(  1+\lambda
+\frac{\gamma}{2},\gamma+1;\zeta\right)  e^{-\zeta/2}e^{i\pi\frac{\gamma}{4}},%
\]%
\[
\widehat{\mathbf\Psi}_{22}=\frac{\Gamma\left(  1+\lambda+\frac{\gamma}{2}\right)
}{\Gamma\left(  \frac{\gamma}{2}-\lambda\right)  }\zeta^{\gamma/2}\psi\left(
1+\lambda+\frac{\gamma}{2},\gamma+1;\zeta\right)  e^{-\zeta/2}e^{-i\pi\frac{\gamma}%
{4}}.
\]
By construction,  $\mathbf\Psi$ satisfies the jumps
relations in ($\Psi2$). Using formulas (7.26), (7.27) and (7.29) from \cite[appendix]{Its07b}, %
we can write explicitly the matrix $\mathbf\Psi$ in all regions.
Since the local behavior of $\psi(a,b;z)$ depends only on the value of the parameter $b$, by construction, all rows of $\widehat{\mathbf{\Psi}}$ have the same asymptotics as $\zeta\to 0$. Hence, it is sufficient to analyze the first row. 

From formulas (13.5.5) and (13.5.12) from \cite{abramowitz/stegun:1972} if follows that for $\zeta\in$
\ding{178}, $\widehat{\mathbf\Psi}$ has the behavior described in
$(\Psi3)$, as $\zeta\rightarrow0$. Indeed,  for $\gamma>0$,
\[
\widehat{\mathbf\Psi}_{11}=\mathcal{O}\left(  \zeta^{\gamma/2}\right),
\quad
\widehat{\mathbf\Psi}_{12}=\mathcal{O}\left(
\zeta^{-\gamma/2}\right) ;
\]
for $\gamma=0$,
\[
\widehat{\mathbf\Psi}_{11}=\mathcal{O}\left(  1\right),
\quad
\widehat{\mathbf\Psi}_{12}=\mathcal{O}\left(  \ln \zeta\right) ;
\]
and for $-1<\gamma<0$,%
\[
\widehat{\mathbf\Psi}_{11}=\mathcal{O}\left(  \zeta^{\gamma
/2}\right),
\quad
\widehat{\mathbf\Psi}_{12}=\mathcal{O}\left(
\zeta^{\gamma/2}\right) .
\] 
Analogously we can check that $\mathbf\Psi$ satisfies ($\Psi3$) in all regions of the plane.

Finally, using formula (13.1.22) from \cite{abramowitz/stegun:1972},%
\[
\begin{vmatrix}
\phi\left(  a,b;\zeta\right)  & \psi\left(  a,b;\zeta\right) \\
\phi^{\prime}\left(  a,b;\zeta\right)  & \psi^{\prime}\left(  a,b;\zeta\right)
\end{vmatrix}
=-\frac{\Gamma\left(  b\right)  e^{\zeta}}{\zeta^{b}\Gamma\left(
a\right)  },%
\]%
as well as the differential relations (13.4.23) and (13.4.10) from \cite{abramowitz/stegun:1972}, we easily get that
\[
\left\vert
\begin{array}
[c]{cc}%
\frac{\Gamma\left(  b-a\right)  }{\Gamma\left(  b\right)  }\zeta^{\frac{b-1}{2}%
}\phi\left(  a,b;\zeta\right)  e^{-\zeta/2}e^{\frac{i\pi\gamma}{4}} & -\zeta^{\frac
{b-1}{2}}\psi\left(  a,b;\zeta\right)  e^{-\zeta/2}e^{-\frac{i\pi\gamma}{4}}\\
\frac{\Gamma\left(  1+a\right)  }{\Gamma\left(  b\right)  }\zeta^{\frac{b-1}{2}%
}\phi\left(  a+1,b;\zeta\right)  e^{-\zeta/2}e^{\frac{i\pi\gamma}{4}} & \frac
{\Gamma\left(  1+a\right)  }{\Gamma\left(  -\left(  1+a-b\right)  \right)
}\zeta^{\frac{b-1}{2}}\psi\left(  a+1,b;\zeta\right)  e^{-\zeta/2}e^{-\frac{i\pi\gamma}%
{4}}%
\end{array}
\right\vert =1.
\]
This implies tha $\det\  \widehat{\mathbf\Psi}  =1$, and, by construction,
 $\det \mathbf \Psi  =1$, which concludes the proof.
\end{proof}
In order to construct the analytic function $\mathbf E_n$ in \eqref{Def-P} we need to study also the asymptotic behavior of $\mathbf \Psi$ at infinity. Let us introduce the notation
\begin{gather}
\upsilon_n\isdef \upsilon_n\left(  \lambda\right)  =\frac{\left(  \lambda
+\frac{\gamma}{2}\right)  _{n}\left(  \lambda-\frac{\gamma}{2}\right)  _{n}%
}{n!},\label{def-an}\\
\tau_{\lambda} \isdef \frac{\Gamma\left(  -\lambda+\frac{\gamma}{2}\right)  }{\left(
-\frac{\gamma}{2}-\lambda\right)  \Gamma\left(  \frac{\gamma}{2}%
+\lambda\right)  }=-\frac{\Gamma\left(  \frac{\gamma}{2}-\lambda\right)
}{\Gamma\left(  \frac{\gamma}{2}+\lambda+1\right)  }; \label{def-tau}%
\end{gather}
observe that
\begin{equation}\label{prop-a1}
\tau_{-\lambda}=\overline{\tau_{\lambda}},\quad
\upsilon_n\left(-\lambda\right)=\overline{\upsilon_n\left(\lambda\right)},\quad \text{and} \quad
\upsilon_{1}=\left(  \lambda^{2}-\frac
{\gamma^{2}}{4}\right)\in \R\,. %
\end{equation}

\begin{lemma}\label{lemma_asymptoticsPsi}
As $\zeta\rightarrow\infty$, $\zeta\in\mathbb{C}\backslash\Gamma$,%
\begin{equation}\label{asymptoticsPsi}
\mathbf \Psi\left(  \zeta\right)  =\left[  \mathbf{I} +\sum_{n=1}^{R-1}\frac{1}{\zeta^{n}}\left(
\begin{array}
[c]{cc}%
\left(  -1\right)  ^{n}\upsilon_n & n\tau_{\lambda}\overline{\upsilon_n}\\
\left(  -1\right)  ^{n}n\overline{\tau_{\lambda}}\upsilon_n & \overline{\upsilon_n}%
\end{array}
\right)  +\OO\left(  \left\vert \zeta\right\vert ^{-R}\right)  \right]
\zeta^{-\lambda\sigma_{3}}e^{\frac{-\zeta\sigma_{3}}{2}} M^{-1}\left(  \zeta\right)
\end{equation}%
 with $\upsilon_n$ defined by \eqref{def-an}, $\tau_{\lambda}$\ defined by \eqref{def-tau}, $\lambda=i\log(c)/\pi$, and%
\[
M\left(  \zeta\right)  \isdef  
\begin{cases}
e^{\frac{\gamma}{4}\pi i\sigma_{3}}e^{-\lambda\pi i\sigma_{3}}, & \frac{\pi
}{2}<\arg\zeta<\pi,\\
e^{-\frac{\gamma}{4}\pi i\sigma_{3}}e^{-\lambda\pi i\sigma_{3}}, & \pi
<\arg\zeta<\frac{3\pi}{2},\\
e^{\frac{\gamma}{4}\pi i\sigma_{3}} 
\begin{pmatrix}
0 & 1\\ 
-1 & 0
\end{pmatrix}  , & -\frac{\pi}{2}<\arg\zeta<0,\\
e^{-\frac{\gamma}{4}\pi i\sigma_{3}} 
\begin{pmatrix}
0 & 1\\
-1 & 0
\end{pmatrix} , & 0<\arg\zeta<\frac{\pi}{2},
\end{cases}
\]%
where we use the main branch of $\zeta^{-\lambda}=e^{-\lambda\log\zeta}$ with
the cut along $i\mathbb{R}_{-}$.
\end{lemma}

\begin{proof}
We use the classical formulas (13.5.1) and (13.5.2) from \cite{abramowitz/stegun:1972} for the confluent hypergeometric
functions. If we take $b=\gamma+1$, and multiply $\phi$ and $\psi$ by $(z^{\gamma
/2}e^{-z/2})$, using \eqref{Def-G-H}, we have that, as $\left\vert z\right\vert
\rightarrow\infty$, 
\begin{equation}\label{for-13.5.1}
G\left(  a,\gamma;z\right)  =\left\{
\begin{array}
[c]{cc}%
\tfrac{\Gamma\left(  \gamma+1\right)  }{\Gamma\left(  a\right)  }%
z^{a-\gamma/2}\left[  \tfrac{1}{z}\left(  1+\sum_{n=1}^{R-1}\tfrac{\left(  \gamma+1-a\right)
_{n}\left(  1-a\right)  _{n}}{n!z^{n}}+\OO\left(  \left\vert z\right\vert
^{-R}\right)  \right)  \right]  e^{z/2}, & \operatorname{Re}z>0,\\
\tfrac{\Gamma\left(  \gamma+1\right)  }{\Gamma\left(  \gamma+1-a\right)
}e^{a\pi i}z^{\gamma/2-a}\left[  1%
+\sum_{n=1}^{R-1}\tfrac{(a)_{n}\left(  a-\gamma\right)  _{n}}{\left(  -1\right)
^{n}n!z^{n}}+\OO\left(  \left\vert z\right\vert ^{-R}\right)  \right]
e^{-z/2}, & \operatorname{Re}z<0,
\end{array}
\right.%
\end{equation}
\begin{equation} \label{for-13.5.2}
H\left(  a,\gamma;z\right)  = z^{\gamma/2-a}\left[  1+\sum_{n=1}^{R-1}\tfrac{(a)_{n}\left(  a-\gamma\right)
_{n}}{\left(  -1\right)  ^{n}n!z^{n}}+\OO\left(  \left\vert z\right\vert
^{-R}\right)  \right]  e^{-z/2}.%
\end{equation}%
Replacing these expansions in the expression for $\mathbf{\Psi}$ for $\zeta\in$ \ding{172}, $\frac{\pi}{2}<\arg\zeta<\frac{3\pi}{4}$ and
$\frac{-\pi}{2}<\arg\left(  e^{-\pi i}\zeta\right)  <\frac{-\pi}{4}$, we get for $\left\vert
\zeta\right\vert \rightarrow\infty$,
\begin{multline*}
\mathbf \Psi\left(  \zeta\right)  =\left(
\begin{array}
[c]{c}%
\zeta^{-\lambda}\left[  1+\sum_{n=1}^{R-1}%
\tfrac{\left(  \lambda+\frac{\gamma}{2}\right)  _{n}\left(  \lambda
-\frac{\gamma}{2}\right)  _{n}}{\left(  -1\right)  ^{n}n!\zeta^{n}}+\OO\left(
\left\vert \zeta\right\vert ^{-R}\right)  \right]  e^{-\zeta/2}e^{\lambda\pi
i}\\
\frac{-\Gamma\left(  1+\lambda+\frac{\gamma}{2}\right)  }{\Gamma\left(
\frac{\gamma}{2}-\lambda\right)  }\zeta^{-1-\lambda}\left[  1+\sum_{n=1}^{R-1}\tfrac{\left(  1+\lambda+\frac{\gamma}{2}\right)
_{n}\left(  1+\lambda-\frac{\gamma}{2}\right)  _{n}}{\left(  -1\right)
^{n}n!\zeta^{n}}+\OO\left(  \left\vert \zeta\right\vert ^{-R}\right)  \right]
e^{-\zeta/2}e^{\lambda\pi i}%
\end{array}
\right. \\
\left.
\begin{array}
[c]{c}%
-\frac{\Gamma\left(  1-\lambda+\frac{\gamma}{2}\right)  }{\Gamma\left(
\frac{\gamma}{2}+\lambda\right)  }\left(  e^{-\pi i}\zeta\right)
^{-1+\lambda}\left[  1+\sum_{n=1}^{R-1}%
\tfrac{\left(  1-\lambda+\frac{\gamma}{2}\right)  _{n}\left(  1-\lambda
-\frac{\gamma}{2}\right)  _{n}}{\left(  -1\right)  ^{n}\left(  -1\right)
^{n}\zeta^{n}n!}+\OO\left(  \left\vert \zeta\right\vert ^{-R}\right)  \right]
e^{\zeta/2}\\
\left(  e^{-\pi i}\zeta\right)  ^{\lambda}\left[  1+\sum_{n=1}^{R-1}\tfrac{\left(  \frac{\gamma}{2}-\lambda\right)
_{n}\left(  -\frac{\gamma}{2}-\lambda\right)  _{n}}{\left(  -1\right)
^{n}\left(  -1\right)  ^{n}\zeta^{n}n!}+\OO\left(  \left\vert \zeta\right\vert
^{-R}\right)  \right]  e^{\zeta/2}%
\end{array}
\right)  e^{-\frac{\gamma\pi i}{4}\sigma_{3}},%
\end{multline*}
which can be rewritten using notation \eqref{def-an}--\eqref{def-tau} as
\begin{multline*}
=\left[  \mathbf{I}+\left(
\begin{array}
[c]{c}%
\left[  \sum_{n=1}^{R-1}\left(  -1\right)  ^{n}%
\tfrac{\upsilon_n}{\zeta^{n}}\right] \\
\overline{\tau_{\lambda}}\left[ \sum_{n=1}^{R-1}\left(
-1\right)  ^{n}\tfrac{n\upsilon_n}{\zeta^{n}}\right]
\end{array}
\right.  \right. \\
\left.  \left.
\begin{array}
[c]{c}%
\tau_{\lambda}\left[ \sum_{n=1}^{R-1}%
\tfrac{n\overline{\upsilon_n}}{\zeta^{n}}\right] \\
\left[  \sum_{n=1}^{R-1}\tfrac{\overline
{\upsilon_n}}{\zeta^{n}}\right]
\end{array}
\right)  +\OO\left(  \left\vert \zeta\right\vert ^{-R}\right)  \right]  \left(
\zeta^{-\lambda}e^{-\zeta/2}e^{\lambda\pi i}e^{-\frac{\gamma\pi i}{4}}\right)
^{\sigma_{3}}.%
\end{multline*}
This yields \eqref{asymptoticsPsi} for $\pi/2<\zeta<3\pi/4$; this expansion is also valid for $\zeta\in$ \ding{173}. A comparison of \eqref{sol Psi2} with \eqref{sol Psi3} shows that the behavior for $\zeta\in$ \ding{174}, $\pi<\arg\zeta<\frac{5\pi}{4}$, can be obtained from the expansion in \ding{173} by multiplying by $e^{i\frac{\gamma}{2}  \pi \sigma_{3}}$, which again yields \eqref{asymptoticsPsi}  for $\pi<\zeta<5\pi/4$.
It is easy to see that asymptotics in \ding{174} is also valid in \ding{175}.

Using \eqref{sol Psi5}, \eqref{for-13.5.1}, \eqref{for-13.5.2} and comparing the expression for $\mathbf{\Psi}$ in \ding{172} and \ding{176},  we conclude that for $\zeta\in$ \ding{176}, $-\frac{\pi}{2}<\arg\zeta<-\frac{\pi}{4}$ ($\frac{\pi}{2}<\arg\left(\zeta\right) e^{\pi i}<\frac{3\pi}{4}$ and $\operatorname{Re}\zeta
>0$), as $\left\vert \zeta\right\vert \rightarrow\infty$, 
\begin{align*}
\mathbf \Psi\left(  \zeta\right)  =& \left(
\begin{array}
[c]{c}%
-\frac{\Gamma\left(  1-\lambda+\frac{\gamma}{2}\right)  }{\Gamma\left(
\frac{\gamma}{2}+\lambda\right)  }\left(  e^{\pi i}\zeta\right)  ^{-1+\lambda
}\left[  1+\sum_{n=1}^{R-1}\tfrac{\left(
1-\lambda+\frac{\gamma}{2}\right)  _{n}\left(  1-\lambda-\frac{\gamma}%
{2}\right)  _{n}}{\left(  -1\right)  ^{n}\left(  -1\right)  ^{n}\zeta^{n}%
n!}+\OO\left(  \left\vert \zeta\right\vert ^{-R}\right)  \right]  e^{-\lambda\pi
i}e^{\zeta/2}\\
 \left(  e^{\pi i}\zeta\right)  ^{\lambda}\left[  1+\sum_{n=1}^{R-1}\tfrac{\left(  \frac{\gamma}{2}-\lambda\right)
_{n}\left(  -\frac{\gamma}{2}-\lambda\right)  _{n}}{\left(  -1\right)
^{n}\left(  -1\right)  ^{n}\zeta^{n}n!}+\OO\left(  \left\vert \zeta\right\vert
^{-R}\right)  \right]  e^{-\lambda\pi i}e^{\zeta/2}%
\end{array}
\right. \\&
\left.
\begin{array}
[c]{c}%
-\zeta^{-\lambda}\left[  1+\sum_{n=1}^{R-1}%
\tfrac{\left(  \lambda+\frac{\gamma}{2}\right)  _{n}\left(  \lambda
-\frac{\gamma}{2}\right)  _{n}}{\left(  -1\right)  ^{n}n!\zeta^{n}}+\OO\left(
\left\vert \zeta\right\vert ^{-R}\right)  \right]  e^{-\zeta/2}\\
-\frac{-\Gamma\left(  1+\lambda+\frac{\gamma}{2}\right)  }{\Gamma\left(
\frac{\gamma}{2}-\lambda\right)  }\zeta^{-1-\lambda}\left[  1+\sum_{n=1}^{R-1}\tfrac{\left(  1+\lambda+\frac{\gamma}{2}\right)
_{n}\left(  1+\lambda-\frac{\gamma}{2}\right)  _{n}}{\left(  -1\right)
^{n}n!\zeta^{n}}+\OO\left(  \left\vert \zeta\right\vert ^{-R}\right)  \right]
e^{-\zeta/2}%
\end{array}
\right)  e^{-\frac{\gamma\pi i}{4}\sigma_{3}}%
\\  =& \left[  \left(
\begin{array}
[c]{ll}%
0 & -1\\
1 & 0
\end{array}
\right)  +\sum_{n=1}^{R-1}\frac{1}{\zeta^{n}}\left(
\begin{array}
[c]{cc}%
n\tau_{\lambda}\overline{\upsilon_n} & -\left(  -1\right)  ^{n}\upsilon_n\\
\overline{\upsilon_n} & -\left(  -1\right)  ^{n}n\overline{\tau_{\lambda}}\upsilon_n%
\end{array}
\right)  +\OO\left(  \left\vert \zeta\right\vert ^{-R}\right)  \right] \\
& \times\left(
\begin{array}
[c]{ll}%
0 & -1\\
1 & 0
\end{array}
\right)  ^{-1}\left(
\begin{array}
[c]{ll}%
0 & -1\\
1 & 0
\end{array}
\right)  \zeta^{\lambda\sigma_{3}}e^{-\frac{\gamma}{4}\pi
i\sigma_{3}}e^{\frac{\zeta}{2}\sigma_{3}}%
\\
= & \left[  \mathbf{I}+\sum_{n=1}^{R-1}\frac{1}%
{\zeta^{n}}\left(
\begin{array}
[c]{cc}%
\left(  -1\right)  ^{n}\upsilon_n & n\tau_{\lambda}\overline{\upsilon_n}\\
\left(  -1\right)  ^{n}n\overline{\tau_{\lambda}}\upsilon_n & \overline{\upsilon_n}%
\end{array}
\right)  +\OO\left(  \left\vert \zeta\right\vert ^{-R}\right)  \right]
\zeta^{-\lambda\sigma_{3}}e^{\frac{\gamma}{4}\pi i\sigma
_{3}}e^{-\frac{\zeta}{2}\sigma_{3}}\left(
\begin{array}
[c]{ll}%
0 & -1\\
1 & 0
\end{array}
\right).  %
\end{align*}%
This expression is valid in  \ding{177} as well. Finally, comparing \eqref{sol Psi6} with \eqref{sol Psi7} we see that the behavior for $\zeta\in$ \ding{178}, $0<\arg\zeta<\frac{\pi}{4}$, corresponds to that in \ding{177} times the constant factor $e^{i\frac{\gamma}{2}  \pi \sigma_{3}}$, which yields \eqref{asymptoticsPsi}.
Since the asymptotics for $\zeta\in$ \ding{179}  is the same than in \ding{178}, this concludes the proof of Lemma.
\end{proof}

Now we are ready to build $\mathbf P^{(1)}$ in \eqref{Def-P}. Using the properties of $\varphi$ 
we define an analytic function $f$ in a neighborhood of $x_0$,
\begin{equation}\label{Def-f}%
f\left(  z\right)  \isdef \begin{cases}
2i\arccos x_0 -2 \log\varphi\left(  z\right)   ,  & \text{
for } \Im z>0, \\
2i\arccos x_0 + 2 \log\varphi\left(  z\right)  ,  & \text{
for } \Im z<0,
\end{cases}
\end{equation}
where we take the main branch of the logarithm. Using that $\varphi_{+}\left(  x\right)  \varphi_{-}\left(  x\right)     =1$ on $( -1,1)$ we  conclude that $f$ can be extended to a holomorphic function in $\C\setminus \left( (-\infty,-1] \cup [1, +\infty) \right)$. For $|z|<1$ we have
\begin{equation} \label{f--0}
f\left(  z\right)  =\dfrac{2i}{\sqrt{1-x_0^2}}\left(z-x_0\right)+\OO\left(  (z-x_0)^2\right)  ,\ \text{as
}z\rightarrow x_0\text{.}%
\end{equation}
Hence, for $\delta>0$ sufficiently small, $f$ is a
conformal mapping of $U_{x_0}$. Moreover, since
\begin{equation} \label{phiplus}
\varphi_+(x)=x+i\sqrt{1-x^2}=e^{i \arccos x}, \quad x\in (-1,1),
\end{equation}
then
\begin{equation}\label{fOnR}
    f(x)=2 i \left(\arccos x_0 - \arccos x \right), \quad x\in (-1,1),
\end{equation}
so that $f$ maps the real interval $\left(  -1,x_0\right)  $ one-to-one onto the purely imaginary interval
$\left(  2i(\arccos x_0 - \pi), 0\right)  $, as well as $\left( x_0 , 1\right)  $ one-to-one onto the purely imaginary interval
$\left( 0, 2i\arccos x_0\right)  $.

We can always deform our contours $\Sigma_k$ close to $z=x_0$ in such a way that
\begin{align*}
f\left( \Sigma_{1}\cap U_{x_0}\right) \subset \Gamma_4, \quad f\left( \Sigma_{2}\cap U_{x_0}\right) \subset \Gamma_6,& \quad f\left( \Sigma_{3}\cap U_{x_0}\right) \subset \Gamma_2, \quad f\left( \Sigma_{4}\cap U_{x_0}\right) \subset \Gamma_8\\
f\left( \Sigma_{5}\cap U_{x_0}\right) \subset \Gamma_3,& \quad f\left( \Sigma_{6}\cap U_{x_0}\right) \subset \Gamma_7.
\end{align*}
With this convention, set
\begin{equation}\label{def zeta}
\zeta  \isdef  nf\left(  z\right), \quad z \in U_{x_0},
\end{equation}%
and, we define
\begin{equation}
\mathbf P^{\left(  1\right)  }\left(  z\right)  \isdef \mathbf \Psi\left(  nf\left(  z\right)
\right), \quad z \in U_{x_0}.  \label{sol-P1}%
\end{equation}
By ($\Psi$1)--($\Psi$3) and \eqref{f--0}, this matrix-valued function has the jumps and the local behavior at $z=x_0$ specified in \eqref{jumpsForP1_1}--\eqref{jumpsForP1_3}. Taking into account the definition \eqref{Def-f} we get that
$$
e^{n f(z)} =\varphi_+^{2n}(x_0) \, \varphi^{\mp 2n}(z), \quad \text{for } \pm \Im z >0,
$$
and for $\left[  nf\left(  z\right)  \right]  ^{\lambda}$ we take the cut along $\left(  -\infty,x_0\right]  $.
Since
\begin{equation*}
    \left[    f\left(  z\right)  \right]  ^{\lambda
}=\left\vert  f\left(  z\right)  \right\vert ^{\lambda}\exp\left(  -\frac{\log
c}{ \pi}\arg\left(   f\left(  z\right)  \right)  \right) ,
\end{equation*}
straightforward computations show that
\begin{equation}\label{jumpFonMinus}
\left[    f\left(  x\right)  \right] _\pm ^{\lambda
}=\begin{cases}
 \left\vert  f\left(  x\right)  \right\vert ^{\lambda} c^{-1/2} , & \text{for } x_0<x<1,\\
\left\vert  f\left(  x\right)  \right\vert ^{\lambda}c^{-1/2 \mp 1} , & \text{for } -1<x<x_0,
\end{cases}
\end{equation}
where we assume the natural orientation of the interval.

In order to satisfy (P$_0$3) above, we define
\begin{equation}
\begin{split}
\mathbf E_{n}\left(  z\right) \isdef\mathbf  N\left(  z\right)  W\left(  z\right)  ^{\sigma_{3}}
\begin{cases}
(nf (z))^{\lambda\sigma_{3}}\varphi_+^{n \sigma_3}(x_0) \, e^{i\frac {\gamma \pi}{4}\sigma_3}c^{\sigma_{3}},&\text{if } z\in Q_+^R, \\[1mm]
\begin{pmatrix}
0 & 1\\
-1 & 0
\end{pmatrix}(nf (z))^{\lambda\sigma_{3}}\varphi_+^{n \sigma_3}(x_0) \, e^{i\frac {\gamma \pi}{4}\sigma_3},& \text{if } z\in Q_-^R,\\[6mm]
(nf (z))^{\lambda\sigma_{3}}\varphi_+^{n \sigma_3}(x_0) \, e^{-i\frac {\gamma \pi}{4}\sigma_3}c^{\sigma_{3}},&\text{if } z\in Q_+^L,\\[1mm]
\begin{pmatrix}
0 & 1\\
-1 & 0
\end{pmatrix}(nf (z))^{\lambda\sigma_{3}}\varphi_+^{n \sigma_3}(x_0) \, e^{-i\frac {\gamma \pi}{4}\sigma_3},& \text{if } z\in Q_-^L.
\end{cases}
\end{split}
 \label{DefinitionEN}%
\end{equation}
By construction, $\mathbf E_{n}$ is analytic in $U_{x_0}\backslash \left(\mathbb{R} \cup \Sigma_5\cup\Sigma_6 \right)$. Furthermore, by (N2) and \eqref{W1}, for $x\in \left(
x_0-\delta ,x_0\right)  \cup (x_0,x_0+\delta ) $,
\begin{equation*}
    \begin{split}
     W_-\left(  x\right)  ^{-\sigma_{3}}  \mathbf N_{-}^{-1}(x)\mathbf N_{+}(x) W_+\left(  x\right)  ^{\sigma_{3}} & =
\begin{pmatrix}
0 & \frac{w_{c,\gamma}(x)}{W_-(x)W_+(x)} \\
-\frac{W_-(x)W_+(x)}{w_{c,\gamma}(x)}   & 0
\end{pmatrix}\\ & =\begin{pmatrix}
0 & c^{\pm 1} \\
-c^{\mp 1}   & 0
\end{pmatrix}, \quad \text{for } \pm x>x_0;
    \end{split}
\end{equation*}%
and, by \eqref{W2}, for $z\in \Sigma_6 \cap U_{x_0}$ (oriented from above to bellow) and for $z\in \Sigma_5 \cap U_{x_0}$ (oriented from bellow to above) we have,
\begin{equation*}
    W_-\left(  z\right)  ^{-\sigma_{3}}  \mathbf N_{-}^{-1}(z)\mathbf N_{+}(z) W_+\left(  z\right)  ^{\sigma_{3}}  =
\begin{pmatrix}
W_+(z)/W_-(z) & 0 \\
0 & W_-(z)/W_+(z)
\end{pmatrix} = \, e^{i\frac {\gamma \pi}{2}\sigma_3}.
\end{equation*}%
From \eqref{jumpFonMinus} and \eqref{DefinitionEN} it follows that
$$
\mathbf E_{n-}^{-1}\left(  z\right) \mathbf E_{n+}\left(  z\right)=\mathbf  I,\quad  \text{for }  z\in U_{x_0}\backslash \left\{x_0 \right\}.
$$
In this form, $x_0$ is the only possible isolated singularity of $\mathbf E_n$ in $U_{x_0}$. The following proposition shows that this is in fact a removable singularity of 
$\mathbf E_{n}$:
\begin{proposition} \label{propE0}
$$
\lim_{z\to x_0} \mathbf E_n(z)= \frac{\sqrt{2}}{2}\, D_{\infty}^{\sigma_{3}} \begin{pmatrix}
e^{-i \arcsin(x_0)/2} & e^{i \arcsin(x_0)/2} \\ -e^{i \arcsin(x_0)/2} & e^{-i \arcsin(x_0)/2}
\end{pmatrix} e^{i\eta_n \sigma_3},
$$
with $\eta_n$ defined by
\begin{equation}\label{defOfEta}
\eta_{n}\isdef \frac{\log c}{\pi}\log\left(  4n\sqrt{1-x_{0}^{2}}\right)  + n\arccos(x_0) - \frac{\gamma \pi}{4} - \Phi \left(x_0\right)
\end{equation}
and $\Phi$ given by \eqref{def_PhinoC}. In particular, $\mathbf E_n$ is analytic in $U_{x_0}$.
\end{proposition}%
\begin{proof}
Since $\mathbf E_n$ is analytic in a neighborhood of $x_0$, it is sufficient to analyze its limit as $z\to x_0$ from the first quarter of the plane, $z\in Q_+^R$.
By \eqref{D-0} and \eqref{f--0},
\begin{align*}
\lim_{\stackrel{z\to x_0}{z\in Q_+^R}} D\left(  z, \Xi_c \right)
 f\left(  z\right)    ^{-\lambda} =\lim_{\stackrel{z\to x_0}{z\in Q_+^R}}  c^{1+ \frac{i}{\pi}\, \log\left(z/2 \right) - \frac{i}{\pi}\, \log\left( f(z) \right)}     =c^{3/2  } \left(4\sqrt{1-x_0^2}\right)^{-\lambda }.
\end{align*}
On the other hand, by \eqref{boundaryValuesFinal} and \eqref{W4} (notice that $w_{1,\gamma}$ defined in \eqref{defwcanalytic} coincides with $w$ defined in \eqref{defwanalytic}),
\begin{align*}
\lim_{\stackrel{z\to x_0}{z\in Q_+^R}} D\left(  z, w \right)
 W\left(  z\right)    ^{-1} = c^{-1/2  } e^{i\Phi(x_0)} e^{i\gamma \pi/2}.
\end{align*}
Summarizing,
\begin{align*}
\lim_{\stackrel{z\to x_0}{z\in Q_+^R}} D\left(  z, w_{c,\gamma} \right)^{-1}
 W\left(  z\right)  f\left(  z\right)    ^{ \lambda}  =\frac{\left(4\sqrt{1-x_0^2}\right)^\lambda}{c} \,  e^{-i\Phi(x_0)-i\gamma \pi/2}.
\end{align*}
By \eqref{sol-N} and \eqref{DefinitionEN}, if $z\in Q_+^R$ ($\Im z>0$),
\begin{equation}\label{expressionForE_n}
  \mathbf   E_n(z) = D_{\infty}^{\sigma_{3}}\,
\mathbf A(z) \, m_{n}(z)^{\sigma_3},
\end{equation}
with
\begin{equation}\label{def_m}
    m_n(z)\isdef \frac{W\left(  z\right)  f\left(  z\right)    ^{ \lambda}}{D\left(  z, w_{c,\gamma} \right)} \, \varphi_+^n(x_0) n^\lambda e^{i\gamma \pi/4} c=e^{\eta_n},
\end{equation}
where $\eta_n$ is defined in \eqref{defOfEta}. 

Gathering the limits computed above and using that
\begin{align*}
\lim_{\stackrel{z\to x_0}{z\in Q_+^R}} A_{11}(z)  = e^{-i\arcsin (x_0)/2}=\lim_{\stackrel{z\to x_0}{z\in Q_+^R}}\overline {A_{12}(z)}
\end{align*}
and by definition of $\eta_n$, the statement follows.
\end{proof}

Therefore, by construction the matrix-valued function $\mathbf P_{x_0}$ given by \eqref{Def-P} satisfies conditions (P$_{0}$1)--(P$_{0}$4). Moreover, it is easy to check that
\begin{equation*}
 \det \mathbf P_{x_0}\left( z\right)  =1 \quad \text{for every }z\in U_{x_0}\backslash\Sigma.
\end{equation*}

At this point all the ingredients are ready to define the final transformation. We take
\begin{equation}
\mathbf R\left(  z\right)  \isdef
\begin{cases}
\mathbf S\left(  z\right)  \mathbf N^{-1}\left(  z\right),  &  z\in\mathbb{C}%
\backslash\left\{  \Sigma\cup U_{-1}\cup U_{x_0}\cup U_{1}\right\} ;\\
\mathbf S\left(  z\right) \mathbf  P_{\zeta}^{-1}\left(  z\right),  &  z\in
U_{j}\setminus \Sigma, \; j\in \{-1, x_0, 1 \}.
\end{cases}
  \label{Def-R}%
\end{equation}
$\mathbf R$ is analytic in the complement to the contours $\Sigma_R$ depicted in Fig.~\ref{fig:lenses3}.%
\begin{figure}[htb]
\centering \begin{overpic}[scale=1.5]%
{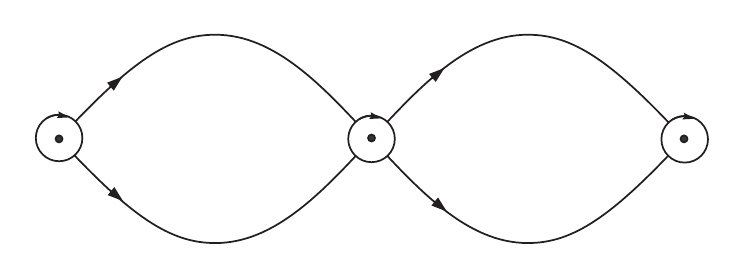}%
     \put(4,12){$\small -1 $}
\put(92,12){$\small 1$}
 \put(49,12){$\small x_0 $}
\end{overpic}
\caption{Contours $\Sigma_R$.}
\label{fig:lenses3}
\end{figure}

Let
\[
\Sigma_{R}^{out}\isdef \Sigma_{R}\setminus\left(  \partial U_{-1}%
\cup\partial U_{x_0}\cup\partial U_{1}\right) .
\]
$\mathbf R$ satisfies the jump relation
$$\mathbf R_{+}(z)=\mathbf R_{-}(z) (\mathbf  I + \mathbf  \Delta(z)),\, z\in \Sigma_{R},$$
with
\[
\mathbf \Delta\left(  s\right)  \isdef
\begin{cases}
\mathbf N\left(  s\right)
\begin{pmatrix}
1 & 0\\
 w_{c,\gamma}\left(  s\right)^{-1}    \varphi\left(  s\right)^{-2n} & 1
\end{pmatrix}
\mathbf  N^{-1}\left(  s\right) - \mathbf{I},  & \text{for }s\in\Sigma_{R}^{out}; \\
\mathbf P_{\zeta}\left(  s\right) \mathbf  N^{-1}\left(  s\right) -\mathbf {I}, & \text{for }s\in\partial
U_{\zeta}, \; j\in \{-1,x_0,1 \}.
\end{cases}
\]%
Standard arguments show that $\mathbf \Delta$ has an asymptotic expansion in powers of $1/n$ of the form
\begin{equation}
\mathbf \Delta\left(  s\right)  \sim\sum_{k=1}^{\infty}\frac{\mathbf \Delta_{k}\left(
s,n\right)  }{n^{k}},\text{ \ \ as }n\rightarrow\infty,\text{ uniformly for
}s\in\Sigma_{R}, \label{DD-n-00}%
\end{equation}
and, for $k\in\mathbb{N}$,
\begin{equation}\label{delta=0}
\mathbf\Delta_{k}\left(  s\right)  =0\text{, \ \ for }s\in\Sigma_{R}^{out}\text{.}%
\end{equation}
So, it remains to determine $\mathbf \Delta_k$ on $\partial U_{x_0}$. Here we find explicitly only the first term, $\mathbf \Delta_1$. %
Using \eqref{sol-N}, \eqref{Def-W}, \eqref{Def-P}, \eqref{Def-f}, \eqref{asymptoticsPsi}, \eqref{def zeta}, \eqref{DefinitionEN} and \eqref{prop-a1},  we obtain
\begin{align*}
\mathbf  \Delta\left(  s\right)
&   =\mathbf  E_{n}(s)\left[  \frac{\left(\lambda^2-\gamma^2/4 \right)}{n f(s)}
\begin{pmatrix}
-1  & \tau_\lambda \\
-\overline{\tau_{ \lambda}}  & 1 %
\end{pmatrix} +\OO\left(  \frac{1}{n^{2}}\right)  \right]  \mathbf E_{n}^{-1}(s), \quad s\in\partial U_{x_0}, \quad n\rightarrow\infty.
\label{comp DD}%
\end{align*}%
Let us define
\begin{equation}\label{DD-1}
  \mathbf \Delta_1\left(  s\right) \isdef \frac{\left(\lambda^2-\gamma^2/4\right)}{  f(s)}\,  \mathbf E_{n}(s)
\begin{pmatrix}
-1  & \tau_\lambda \\
-\overline{\tau_{ \lambda}}  & 1 %
\end{pmatrix}  \mathbf   E_{n}^{-1}(s), \quad s\in\partial U_{x_0}.
\end{equation}
Using that by \eqref{DefinitionEN},
$$
\mathbf E_n(s)= \mathbf F(s) \,  \left(\varphi_+(x_0)^n n^\lambda  \right)  ^{ \sigma_{3}}= \mathbf F(s) \,  \left(e^{in\arccos (x_0)} c^{\frac{i}{\pi}\, \log n}  \right)  ^{ \sigma_{3}},
$$
with%
\begin{equation*} 
\begin{split}
\mathbf F\left(  s\right) \isdef  \begin{cases}
 \mathbf N\left(  s\right)  W\left(  s\right)  ^{\sigma_{3}} c^{ \sigma_{3}} e^{\pm \frac{\gamma \pi}{4}\sigma_3}  f\left( s\right)  ^{\lambda\sigma_{3}}  , &\text{if } \Im s>x_0, \\
\mathbf  N\left(  s\right)  W\left( s\right)  ^{\sigma_{3}}   \begin{pmatrix}
0 & 1\\
-1 & 0
\end{pmatrix} e^{\pm \frac{\gamma \pi}{4}\sigma_3}  f\left(  s\right)   ^{ \lambda\sigma_{3}}, & \text{if } \Im s<x_0,
\end{cases}
\end{split}
\end{equation*}%
where we take $\pm$ for $\pm \Re s>x_0$, we conclude that, for $s\in \partial U_{x_0}$,
$\Delta_{1}\left(  z,n\right)  $ is uniformly bounded in
$n$; indeed, $\mathbf F$ does not depend on $n$ and
$$\left\vert e^{in\arccos x_{0}}c^{\frac{i\log n}{2\pi}}\right\vert =1, \quad
\forall n\in \N.$$%
So $\mathbf \Delta_{1}$ in \eqref{DD-1} is genuinely the first coefficient in the expansion \eqref{DD-n-00}.

Similar analysis can be performed for $\mathbf \Delta_k\left(  \cdot, n\right)$, $k\geq 2$, taking higher order terms in the expansion of $\mathbf \Psi$ in \eqref{asymptoticsPsi}.

The explicit expression \eqref{asymptoticsPsi} and the local behavior of $f$ show that $\mathbf \Delta_{1}\left(  s, n\right)  $ has an analytic continuation to $U_{x_0}$ except for $x_0$, where it has a simple pole. Again, similar conclusion is valid for other $\mathbf \Delta_{k}\left(  s, n\right)  $, except that now the pole has order $k$.

Like in \cite[Theorem 7.10]{MR2001f:42037} we obtain from \eqref{DD-n-00} that
\begin{equation}
\mathbf R\left(  z\right)  \sim\mathbf {I}+\sum_{j=1}^{\infty}\frac{\mathbf R^{\left(
j\right)  }\left(  z, n\right)  }{n^{j}},\text{ \ \ as }n\rightarrow
\infty\text{,} \label{R-n--00}%
\end{equation}
uniformly for $z\in\mathbb{C}\backslash\left\{  \partial U_{ -1}%
\cup\partial U_{x_0}\cup\partial U_{ 1}\right\}  $, where each
$\mathbf R^{\left(  j\right)  }\left(  z\right)  $ is analytic, uniformly bounded in $n$, and%
\[
\mathbf R^{\left(  j\right)  }\left(  z, n\right)  =\OO\left(  \frac{1}{z}\right)  \quad \text{as }z\rightarrow\infty\text{.}%
\]
Since $\mathbf R^{\left(  1\right)  }$ is analytic in the complement of $\partial
U_{-1}\cup\partial U_{x_0}\cup\partial U_{1}$ (see \eqref{delta=0}) and vanishes at infinity, by Sokhotskii-Plemelj formulas,
\[
\mathbf R^{\left(  1\right)  }\left(  z,n\right)  =\frac{1}{2\pi i}\int_{\partial
U_{-1}\cup\partial U_{x_0}\cup\partial U_{1}}\frac
{\mathbf \Delta_{1}\left(  s, n\right)  }{s-z}\, ds.
\]
Recall that $\mathbf \Delta_{1}$ can be extended analytically inside $U_j$'s with simple poles at $\pm 1 $ and $x_0$; let us denote by $A^{\left(  1\right)  }\left(  n\right)  $,
$B^{\left(  1\right)  }\left(  n\right)  $ and $C^{\left(  1\right)  }\left(
n\right)  $ the residues of $\mathbf \Delta_{1}(\cdot, n)$ at $1$, $-1$ and $x_0$, respectively. Then the residue calculus gives
\begin{equation}
\mathbf R^{\left(  1\right)  }\left(  z,n\right)  =
\begin{cases}
\dfrac{A^{\left(  1\right)  }\left(  n\right)  }{z-1}+\dfrac{B^{\left(
1\right)  }\left(  n\right)  }{z+1}+\dfrac{C^{\left(  1\right)  }\left(
n\right)  }{z-x_0}, & \text{for }z\in\mathbb{C}\backslash\left\{  U_{ -1}\cup
U_{ x_0}\cup U_{ 1}\right\}  ;\\[3mm]
\dfrac{A^{\left(  1\right)  }\left(  n\right)  }{z-1}+\dfrac{B^{\left(
1\right)  }\left(  n\right)  }{z+1}+\dfrac{C^{\left(  1\right)  }\left(
n\right)  }{z-x_0}-\mathbf \Delta_{1}\left(  z,n\right) , & \text{for }z\in
U_{ -1}\cup U_{x_0}\cup U_{ 1}  .
\end{cases}
  \label{comp-R1}%
\end{equation}
Residues $A^{\left(  1\right)  }\left(  n\right)  $ and $B^{\left(  1\right)  }\left(  n\right)  $ are in fact independent of $n$; they have been determined in \cite[Section 8]{MR2087231}:
\begin{equation}
\begin{split}
\label{A-B}
A^{\left(  1\right)  }\left(  n\right) & =A^{(1)} =\frac{4\alpha^2-1}{16}\, D_{\infty}^{\sigma_3}
    \begin{pmatrix}
        -1 & i \\
        i & 1
    \end{pmatrix}D_{\infty}^{-\sigma_3},
\\
B^{\left(  1\right)  }\left(  n\right) & =B^{(1)} = 
    \frac{4\beta^2-1}{16}\, D_{\infty}^{\sigma_3}
    \begin{pmatrix}
        1 & i\\
        i & -1
    \end{pmatrix}
    D_{\infty}^{-\sigma_3}
\end{split}
\end{equation}
(notice however that the constant $D_\infty$ is different with respect to \cite{MR2087231}). The value of the remaining residue $C^{\left(
1\right)  }\left(  n\right)  $ is given in the following

\begin{proposition}
\label{prop coef C} If we denote 
\[
C^{\left(  1\right)  }\left(  n\right)  = 
\begin{pmatrix}
C_{11}^{(1)}(n) & C_{12}^{(1)}(n)\\
C_{21}^{(1)}(n) & C_{11}^{(1)}(n)
\end{pmatrix}
\]%
then the entries are given explicitly by:%
\begin{equation}
C_{11}^{\left(  1\right)  }\left(  n\right)  =-\left(  \frac{\log^{2}c}%
{2\pi^{2}}+\frac{\gamma^{2}}{8}\right)  x_{0}+\sqrt{\frac{\log^{2}c}{4\pi
^{2}}+\frac{\gamma^{2}}{16}}\sin\theta_{n} \label{C11 n}%
\end{equation}%
\begin{equation}
C_{12}^{\left(  1\right)  }\left(  n\right)  =iD_{\infty}^{2}\left(
\frac{\log^{2}c}{2\pi^{2}}+\frac{\gamma^{2}}{8}-\sqrt{\frac{\log^{2}c}%
{4\pi^{2}}+\frac{\gamma^{2}}{16}}\cos\left(  \arcsin\left(  x_{0}\right)
-\theta_{n}\right)  \right)  \label{C12 n}%
\end{equation}%
\begin{equation}
C_{21}^{\left(  1\right)  }\left(  n\right)  =\frac{i}{D_{\infty}^{2}}\left(
\frac{\log^{2}c}{2\pi^{2}}+\frac{\gamma^{2}}{8}+\sqrt{\frac{\log^{2}c}%
{4\pi^{2}}+\frac{\gamma^{2}}{16}}\cos\left(  \arcsin\left(  x_{0}\right)
+\theta_{n}\right)  \right)  \label{C21 n}%
\end{equation}%
\begin{equation}
C_{22}^{\left(  1\right)  }\left(  n\right)  =\left(  \frac{\log^{2}c}%
{2\pi^{2}}+\frac{\gamma^{2}}{8}\right)  x_{0}-\sqrt{\frac{\log^{2}c}{4\pi
^{2}}+\frac{\gamma^{2}}{16}}\sin\theta_{n} \label{C22 n}%
\end{equation}
where
\begin{equation}\label{definition-theta_n}
\theta_{n}=2\eta_{n}+\varsigma,
\end{equation}
with $\eta_n$ defined by \eqref{defOfEta} and $\varsigma=-2\arg\Gamma\left(
\frac{\gamma}{2}+\lambda\right)  -\arg\left(  \frac{\gamma}{2}+\lambda\right)
$.
\end{proposition}

\begin{proof}

Taking into account \eqref{f--0} and \eqref{DD-1} we conclude that the
residue $C^{\left(  1\right)  }\left(  n\right)  $ of $\mathbf \Delta_{1}\left(
z,n\right)  $ at $z=x_{0}$ is given by%
\begin{equation}
C^{\left(  1\right)  }\left(  n\right)  =\frac{\left(  \lambda^{2}%
-\gamma^2/4\right)  \sqrt{1-x_{0}^{2}}}{2i}\mathbf E_{n}\left(
x_{0}\right)  \left(
\begin{array}
[c]{cc}%
-1 & \tau_{\lambda}\\
-\overline{\tau_{\lambda}} & 1
\end{array}
\right) \mathbf  E_{n}^{-1}\left(  x_{0}\right).  \label{C-1-n  2}%
\end{equation}
Since $\mathbf E_{n}$ is analytic in a neighborhood of $x_{0}$ (see Proposition \ref{propE0}), 
\begin{equation}\label{dem Delta 1}
\mathbf E_{n}\left(  x_{0}\right) =\lim\limits_{\substack{z\rightarrow x_{0}\\z\in Q_+^R}}\mathbf E_{n}\left(
z\right)  =\frac{1}{\sqrt{2}\sqrt[4]{1-x_{0}^{2}}}D_{\infty}^{\sigma_{3}%
}\left(
\begin{array}
[c]{cc}%
e^{-i\arcsin\left(  x_{0}\right)  /2} & e^{i\arcsin\left(  x_{0}\right)  /2}\\
-e^{i\arcsin\left(  x_{0}\right)  /2} & e^{-i\arcsin\left(  x_{0}\right)  /2}%
\end{array}
\right)  e^{i\eta_{n}\sigma_{3}},%
\end{equation}
so that
\begin{equation}
\mathbf E_{n}^{-1}\left(  x_{0}\right)  =\frac{1}{\sqrt{2}\sqrt[4]{1-x_{0}^{2}}%
}e^{-i\eta_{n}\sigma_{3}}\left(
\begin{array}
[c]{cc}%
e^{-i\arcsin\left(  x_{0}\right)  /2} & -e^{i\arcsin\left(  x_{0}\right)
/2}\\
e^{i\arcsin\left(  x_{0}\right)  /2} & e^{-i\arcsin\left(  x_{0}\right)  /2}%
\end{array}
\right)  D_{\infty}^{-\sigma_{3}}\text{.} \label{dem C1n -1}%
\end{equation}
From \eqref{dem Delta 1} we obtain
\begin{multline*}
C^{\left(  1\right)  }\left(  n\right)  = \frac{\left(  \lambda^{2}-\gamma^{2}/4\right)  }{4i}D_{\infty
}^{\sigma_{3}}\left(
\begin{array}
[c]{cc}%
e^{-i\arcsin\left(  x_{0}\right)  /2} & e^{i\arcsin\left(  x_{0}\right)  /2}\\
-e^{i\arcsin\left(  x_{0}\right)  /2} & e^{-i\arcsin\left(  x_{0}\right)  /2}%
\end{array}
\right)  e^{i\eta_{n}\sigma_{3}}\\
\times\left(
\begin{array}
[c]{cc}%
-1 & \tau_{\lambda}\\
-\overline{\tau_{\lambda}} & 1
\end{array}
\right)  e^{-i\eta_{n}\sigma_{3}}\left(
\begin{array}
[c]{cc}%
e^{-i\arcsin\left(  x_{0}\right)  /2} & -e^{i\arcsin\left(  x_{0}\right)
/2}\\
e^{i\arcsin\left(  x_{0}\right)  /2} & e^{-i\arcsin\left(  x_{0}\right)  /2}%
\end{array}
\right)  D_{\infty}^{-\sigma_{3}}.%
\end{multline*}%
Using formulas (6.1.28), (6.1.23) and (4.3.2) from \cite{abramowitz/stegun:1972} and \eqref{def-tau} we can rewrite%
$$
\tau_{\lambda}   =%
-\frac{\overline{\Gamma\left(  \frac{\gamma}{2}+\lambda\right)  }}{\left(
\frac{\gamma}{2}+\lambda\right)  \Gamma\left(  \frac{\gamma}{2}+\lambda
\right)  }
 =%
-\left\vert \frac{\overline{\Gamma\left(  \frac{\gamma}{2}+\lambda\right)
}}{\left(  \frac{\gamma}{2}+\lambda\right)  \Gamma\left(  \frac{\gamma}%
{2}+\lambda\right)  }\right\vert e^{i\varsigma}
 =
-\frac{e^{i\varsigma}}{\sqrt{\gamma^{2}/4+\left\vert \lambda\right\vert ^{2}}},%
$$%
where
\begin{align*}%
\varsigma &  =\arg\left(  \frac{\overline{\Gamma\left(  \frac{\gamma}%
{2}+\lambda\right)  }}{\left(  \frac{\gamma}{2}+\lambda\right)  \Gamma\left(
\frac{\gamma}{2}+\lambda\right)  }\right).
\end{align*}%
Then,
\begin{equation*}
C^{(1)}(n)=\frac{\left(  \lambda^{2}-\gamma^{2}/4\right)  }{2}D_{\infty
}^{\sigma_{3}}
\begin{pmatrix}
x_{0}-\frac{\sin\theta_{n}}{\sqrt{\gamma^{2}/4+\left\vert \lambda\right\vert ^{2}}}%
 & -i+i\frac{\cos\left(  \arcsin\left(  x_{0}\right)  -\theta_{n}\right)}{\sqrt{\gamma^{2}/4+\left\vert \lambda\right\vert
^{2}}} \\
-i-i\frac{\cos\left(
\arcsin\left(  x_{0}\right)  +\theta_{n}\right) }{\sqrt{\gamma^{2}/4+\left\vert \lambda\right\vert ^{2}}} & -x_{0}-\frac{\sin\left(  -\theta
_{n}\right)}%
{\sqrt{\gamma^{2}/4+\left\vert \lambda\right\vert ^{2}}}
\end{pmatrix}
 D_{\infty}^{-\sigma_{3}}.%
\end{equation*}
We can simplify this expression using that $\left(  \lambda^{2}-\frac{\gamma^{2}}{4}\right)  %
=-\left(\frac{\log^{2}c}{\pi^{2}}+\frac{\gamma^{2}}{4}\right)  %
=-\left(  \sqrt{\frac{\gamma^2}{4}+\left\vert \lambda\right\vert ^{2}}\right)  ^{2}$,  and this settles the proof.
\end{proof}

\section{Proof of Theorem 1}

Unraveling  the transformations $\mathbf Y\rightarrow \mathbf T\rightarrow \mathbf S\rightarrow \mathbf R$ we can obtain an expression for $\mathbf Y$.
Repeating the arguments in \cite{MR2000g:47048} (see also \cite{MR2087231},  \cite[Section 3]{FMS} and \cite{Van:2003}) we see that the recurrence coefficients \eqref{RecRelphin} are given by
\begin{align}
a_{n}^{2}  &%
=\lim\limits_{z\rightarrow\infty}\left(  -\frac{D_{\infty}^{2}%
}{2i}+z\mathbf R_{12}\left(  z,n\right)  \right)  \left(  z\mathbf R_{21}\left(  z,n\right)
+\frac{1}{2iD_{\infty}^{2}}\right)  \text{,}\label{def an}\\
b_{n}  &  =\lim\limits_{z\rightarrow\infty}z \left(  1-\mathbf  R_{11}\left(
z,n+1\right)  \mathbf R_{22}\left(  z,n\right)  \right)  \text{.} \label{def bn}%
\end{align}
Taking into account the expression for $\mathbf R^{(1)}$ in \eqref{comp-R1}, as well as \eqref{A-B} and Proposition \ref{prop coef C}, we obtain for $a_n$:
\begin{equation*}
a_{n}^2 =\frac{1}{4}-\frac{1}{n}\sqrt{\frac{\gamma^{2}}{16}+\frac{\log^{2}c}{4\pi^{2}}}%
\cos\left(  \arcsin x_{0}\right)  \cos\left(  \theta_{n}\right)  +O\left(
\frac{1}{n^{2}}\right), \quad  n\rightarrow\infty\text{,}%
\end{equation*}%
where $\theta_{n}$ is given by \eqref{definition-theta_n}. It also can be rewritten in the form
\begin{equation*}
\theta_{n}    =\frac{2\log c}{\pi}\log\left(  4n\sqrt{1-x_{0}^{2}}\right)
+  2n  \arccos x_{0} - \Theta,
\end{equation*}%
with $\Theta$ defined in \eqref{definitionTheta}. 
This proves \eqref{asymptotics An}.

Analogously,
\begin{equation*}
b_n=-\frac{\sqrt{\frac{\log^{2}c}{4\pi^{2}}+\frac{\gamma^{2}}{16}}\left(
\sin\theta_{n+1}-\sin\theta_{n}\right)  }{n}+O\left(  \frac{1}{n^{2}}\right)
, \quad n \rightarrow\infty.   
\end{equation*}%
By \eqref{definition-theta_n},%
$$
 \theta_{n+1}-\theta_n= 2\arccos (x_0) + 2 \frac{\log c}{\pi}\, \log\left( 1+\frac{1}{n}\right)   ,
$$%
and
\begin{align*}
\sin\theta_{n+1}-\sin\theta_{n} &  %
=\sin\left(  \theta_{n}+2\arccos
x_{0}\right)  -\sin\theta_{n}+\OO\left(\frac{1}{n}\right) \\
& =2\cos\left(  \theta_{n}+\arccos x_{0}\right)  \sin\left(  \arccos
x_{0}\right)+\OO\left(\frac{1}{n}\right) , \quad  n \rightarrow\infty.
\end{align*}%
Thus we obtain%
\begin{equation*}
b_{n}  =-\frac{1}{n}\sqrt{\frac{\log^{2}c}{\pi^{2}}+\frac{\gamma^2}{4}}\sin\left(
\arccos x_{0}\right)  \cos\left(  \theta_{n}+\arccos x_{0}\right)  +O\left(
\frac{1}{n^{2}}\right),
\end{equation*}%
which proves \eqref{asymptotics Bn}.

\section*{Acknowledgements}

AMF is partially supported by Junta de Andaluc\'{\i}a, grants FQM-229, FQM-481, and P06-FQM-01735, as well as by the Ministry of Science and Innovation
of Spain (project code MTM2008-06689-C02-01). VPS is spponsored by FCT (Portugal), under contract/grant
SFRH/BD/29731/2006.

We are grateful to Alexei Borodin for driving our attention to reference \cite{Borodin:2001xr} and for stimulating discussions.

\def\cprime{$'$}


\begin{thebibliography}{10}
\expandafter\ifx\csname url\endcsname\relax
  \def\url#1{\texttt{#1}}\fi
\expandafter\ifx\csname urlprefix\endcsname\relax\def\urlprefix{URL }\fi

\bibitem{abramowitz/stegun:1972}
M.~Abramowitz, I.~A. Stegun, Handbook of Mathematical Functions, Dover Publ.,
  New York, 1972.

\bibitem{MR2001m:05258a}
J.~Baik, P.~Deift, K.~Johansson, On the distribution of the length of the
  second row of a {Y}oung diagram under {P}lancherel measure, Geom. Funct.
  Anal. 10 (4) (2000) 702--731.

\bibitem{Borodin:2001xr}
A.~Borodin, G.~Olshanski, Infinite random matrices and ergodic measures, 
Comm. Math. Phys. 223 (1) (2001) 87--123.

\bibitem{Deift:oe}
P.~Deift, A.~Its, I.~Krasovsky,  Asymptotics of {T}oeplitz, {H}ankel, and {T}oeplitz$+${H}ankel determinants with {F}isher-{H}artwig singularities, preprint Arxiv:0905.0443.

\bibitem{MR2001f:42037}
P.~Deift, T.~Kriecherbauer, K.~T.-R. McLaughlin, S.~Venakides, X.~Zhou, Strong
  asymptotics of orthogonal polynomials with respect to exponential weights,
  Comm. Pure Appl. Math. 52 (12) (1999) 1491--1552.

\bibitem{MR98b:35155}
P.~Deift, S.~Venakides, X.~Zhou, New results in small dispersion {K}d{V} by an
  extension of the steepest descent method for {R}iemann-{H}ilbert problems,
  Internat. Math. Res. Notices (6) (1997) 286--299.

\bibitem{MR94d:35143}
P.~Deift, X.~Zhou, A steepest descent method for oscillatory
  {R}iemann-{H}ilbert problems. {A}symptotics for the {M}{K}d{V} equation, Ann.
  of Math. 137 (2) (1993) 295--368.

\bibitem{MR2000g:47048}
P.~A. Deift, Orthogonal polynomials and random matrices: a {R}iemann-{H}ilbert
  approach, New York University Courant Institute of Mathematical Sciences, New
  York, 1999.

\bibitem{MR96d:34004}
P.~A. Deift, X.~Zhou, Asymptotics for the {P}ainlev\'e {I}{I} equation, Comm.
  Pure Appl. Math. 48 (3) (1995) 277--337.

\bibitem{Fokas92}
A.~Fokas, A.~Its, A.~Kitaev, The isomonodromy approach to matrix models in {2D}
  quantum gravity, Comm. Math. Phys. 147 (1992) 395--430.

\bibitem{FMS}
A.~Foulqui\'e Moreno, A.~Mart\'{\i}nez-Finkelshtein, V.L.~Sousa, Asymptotics of orthogonal
polynomials for a weight with a jump on $[-1,1]$, preprint (2009).

\bibitem{Gakhov:90}
F.~D. Gakhov, Boundary value problems, Dover Publications Inc., New York, 1990,
  translated from the Russian, Reprint of the 1966 translation.

\bibitem{Its07b}
A.~Its, I.~Krasovsky, Hankel determinant and orthogonal polynomials for the
  gaussian weight with a jump, Contemp. Math. 458 (2008) 215-–247.

\bibitem{MR2087231}
A.~B.~J. Kuijlaars, K.~T.-R. McLaughlin, W.~Van~Assche, M.~Vanlessen, The
  {R}iemann-{H}ilbert approach to strong asymptotics for orthogonal polynomials
  on {$[-1,1]$}, Adv. Math. 188 (2) (2004) 337--398.

\bibitem{Magnus1995}
A.~P. Magnus, Asymptotics for the simplest generalized {J}acobi polynomials
  recurrence coefficients from {F}reud's equations: numerical explorations,
  Ann. Numer. Math. 2 (1995) 311--325.

\bibitem{szego:1975}
G.~Szeg{\H{o}}, Orthogonal Polynomials, vol.~23 of Amer.\ Math.\ Soc.\ Colloq.\
  Publ., 4th ed., Amer.\ Math.\ Soc., Providence, {RI}, 1975.

\bibitem{Van:2003}
M.~Vanlessen, Strong asymptotics of the recurrence coefficients of orthogonal
polynomials associated to the generalized Jacobi weight, J. Approx. Theory 125 (2003) 198-–237.

\end{thebibliography}
\end{document}